\documentclass[11pt]{article}

\usepackage[margin=1.5in]{geometry}
\usepackage{amsmath,amsthm,amsfonts,amssymb}
\usepackage[colorlinks=true,linkcolor=blue,citecolor=blue,urlcolor=blue]{hyperref}
\usepackage{graphicx}
\usepackage{float}
\usepackage{latexsym}
\usepackage{graphics}
\usepackage{xspace}
\usepackage{color}
\newcommand{\AR}{R} 
\newcommand{\TT}{L}

\def\calN{{\mathcal N}}

\numberwithin{equation}{section}

\newtheorem{theorem}{Theorem}[section]
\newtheorem{defn}[theorem]{Definition}
\newtheorem{lemma}[theorem]{Lemma}
\newtheorem{prop}[theorem]{Proposition}

\newtheorem{conjecture}[theorem]{Conjecture}

\newtheorem{remark}[theorem]{Remark}
\newtheorem{cor}[theorem]{Corollary}

\def\R{{\mathbb R}}
\def\N{{\mathbb N}}
\def\C{{\cal C}}
\def\Z{{\mathbb Z}}
\def\1{{1\!\!1}}
\def\E{{\mathbb E}}
\def\P{{\mathbb P}}

\def\T{\mathbb{T}}

\def\ind{\mathbb{I}}

\def\cal{\mathcal}

\def\leaves{\partial \T_{n,\Delta}}
\def\mean{:=}
\def\sm{\setminus}

\newcommand{\ba}{\begin{array}}
\newcommand{\ea}{\end{array}}
\newcommand{\bea}{\begin{eqnarray}}
\newcommand{\eea}{\end{eqnarray}}
\newcommand{\beas}{\begin{eqnarray*}}
\newcommand{\eeas}{\end{eqnarray*}}
\newcommand{\be}{\begin{equation}}
\newcommand{\ee}{\end{equation}}
\newcommand{\bt}{\begin{theorem}}
\newcommand{\et}{\end{theorem}}
\newcommand{\bc}{\begin{center}}
\newcommand{\ec}{\end{center}}
\newcommand{\ben}{\begin{enumerate}}
\newcommand{\een}{\end{enumerate}}

\newcommand{\ei}{\end{itemize}}

\newcommand{\ds}{\displaystyle}

\newcommand{\ve}{\varepsilon}

\newcommand{\G}{\mathbb{G}}
\newcommand{\Hg}{\mathbb{H}}
\newcommand{\pr}{\mathbb{P}}
\newcommand{\Pol}{\mathcal{P}}
\newcommand{\mbx}{\mbox{\bf x}}
\newcommand{\mbX}{\mbox{\bf X}}
\newcommand{\mby}{\mbox{\bf y}}

\newcommand{\mbzero}{\mbox{\bf 0}}
\newcommand{\mbone}{\mbox{\bf 1}}
\definecolor{Red}{rgb}{1,0,0}
\definecolor{Blue}{rgb}{0,0,1}

\newcommand{\rn}{m}
\newcommand{\meas}{\nu}
\newcommand{\fmeas}{\mu}
\newcommand{\hcmeas}{\mathbb{P}}
\newcommand{\dmeasl}{\mu^{(2)}_{\lambda}}
\newcommand{\dmeaslM}{\mu^{(M+1)}_{\lambda}}
\newcommand{\ex}{z}
\newcommand{\prodmu}{\fmeas^{\otimes_{|V|}}}
\newcommand{\fstar}{\bar{F}}

\newcommand{\ignore}[1]{\relax}

\author{
{\sf David Gamarnik}\thanks{MIT; e-mail: {\tt gamarnik@mit.edu}. Research supported  by the NSF grant CMMI-1031332.}
\and
{\sf Kavita Ramanan}\thanks{Brown University; e-mail: {\tt
    kavita\_ramanan@brown.edu} Research supported in part by NSF
  grants CMMI-1407504 and NSF DMS-1713032}
}

\begin{document}

\title{Uniqueness of Gibbs Measures for Continuous Hardcore Models} 

\date{}
 
\maketitle

\begin{abstract}
We formulate a  continuous version of the well known  discrete hardcore (or
independent set) model on a locally finite graph, parameterized by the
so-called activity parameter $\lambda > 0$.   In this 
version,   the state or ``spin value'' $x_u$ of any node $u$ of the
graph lies in the interval $[0,1]$,  the hardcore constraint $x_u + x_v
\leq 1$ is satisfied  for every edge $(u,v)$ of the graph, 
and the space of feasible configurations is  given by a convex polytope.        
When the graph is a regular   tree, 
we show  that there is a unique Gibbs measure associated to each activity parameter $\lambda>0$.  
Our result shows that, in contrast to the standard discrete hardcore
model, the continuous hardcore model does not exhibit a phase
transition on the infinite regular tree. 
 We also consider a family of continuous models that interpolate
 between the discrete and continuous hardcore models on a regular tree
 when $\lambda = 1$ and show
  that each member of the family has a unique Gibbs measure, even
  when the discrete model does not. 
  In each case, the proof  entails the 
 analysis of an associated  Hamiltonian dynamical
system that describes a certain 
 limit of the marginal distribution at a node. 
Furthermore, given any sequence of regular graphs with fixed degree
and girth diverging to infinity, 
we apply our results to compute the asymptotic limit of  suitably normalized volumes of the 
corresponding sequence of convex polytopes of feasible configurations.  In particular, this yields 
an approximation for the partition function of the continuous hard core model on a regular graph 
with large girth in the case $\lambda = 1$. 
\end{abstract}

\noindent 
{\bf Key Words. }  Hardcore model, independent set, Gibbs measures, phase transition,
partition function,  linear programming polytope,
volume computation, convex polytope,
  computational hardness, regular graphs \\
{\bf 2010 Mathematics Subject Classification. }  Primary:  60K35, 82B20; Secondary: 82B27, 68W25 

\section{Introduction}

\subsection{Background and Motivation. } 
\label{subs-motivation}

The (discrete) hardcore model, also commonly called the independent set model, 
is a widely studied model in statistical mechanics 
as well as combinatorics and theoretical
computer science. 
The model  defines a family of probability measures on 
configurations on a finite or infinite (but locally finite)  graph
$\G$, parameterized by the so-called {\em activity}  $\lambda > 0$. 
On a finite graph $\G$, with node set $V$ and edge set $E$, the hardcore
probability measure with parameter $\lambda > 0$ is supported on the
collection  of independent sets of the graph $\G$ and the probability
of an independent set $I\subset V$ is  proportional to $\lambda^{|I|}$, where
$|I|$ denotes the size of the independent set.   Equivalently, the
hardcore  
probability measure can be thought of as being supported on the set of
configurations $\mbx = (x_u, u \in V)\in \{0,1\}^V$ that satisfy  the hardcore
constraint  $x_u+x_v\le 1$ for every edge $(u,v) \in E$,  with the probability 
of any feasible configuration $\mbx$ being proportional to $\lambda^{\sum_u x_u}$.   
The equivalence between the two formulations follows from the
observation that given any  hardcore configuration $\mbx$, the set
$I = \{u \in V: x_u = 1\}$ is an independent set, and $\sum_{u \in V}
x_u = |I|$. 
The constructed probability measure is a Gibbs measure, in the sense
that it satisfies a certain spatial Markov property \cite{GeorgyGibbsMeasure,SimonLatticeGases,barvinok2017combinatorics}.
On an infinite graph $\G$, the definition of the hardcore Gibbs measure is no longer explicit. Instead, it is defined
implicitly as a measure that has certain specified  conditional distributions 
on finite subsets of the graph, given the configuration on the complement.   Thus, in
contrast to the case of finite graphs, on infinite graphs, neither
existence nor uniqueness of a Gibbs measure is {\em a priori} guaranteed. 
  While existence can be generically shown for a large
class of models,  uniqueness may fail to hold.    When there are 
multiple Gibbs measures for some parameter,  the model is said to
exhibit a phase transition 
\cite{GeorgyGibbsMeasure,SimonLatticeGases}.

The  standard discrete hardcore model on a regular tree is known to exhibit a phase
transition.  Indeed, 
it was shown  in~\cite{Spitzer75,Zachary83,KellyHardCore} 
that there is a unique  hardcore Gibbs measure on an  infinite $(\Delta+1)$-regular tree
$\T_\Delta$ (i.e., a tree in which  every node has degree $\Delta+1$) 
 if and only if $\lambda\le \lambda_c (\Delta) \mean  \Delta^\Delta/(\Delta-1)^{\Delta+1}$.
In particular, for $\lambda$ in this range, the model exhibits a
certain correlation decay property, whereas when $\lambda>\lambda_c(\Delta)$ 
the model exhibits long-range dependence. Roughly speaking, the correlation decay property says that
the random variables $X_u$ and $X_v$ distributed according to the
marginal of the Gibbs measure at nodes $u$ and $v,$ respectively,
become asymptotically independent as the graph-theoretic distance
between $u$ and $v$ tends to infinity. This property
is known to be equivalent to  uniqueness of the Gibbs measure~\cite{GeorgyGibbsMeasure,SimonLatticeGases}.
The phase transition result above
was recently extended to a  generalization of the
hardcore 
model, which is defined on configurations $\mbx \in \{0,1\ldots,
M\}^V$, for some integer $M$,   that
satisfy 
the hardcore constraint $x_u + x_v \leq M$ for $(u,v) \in E$;  the usual hardcore model is recovered by
setting $M=1$. 
Specifically, it was shown in
\cite{RamananSenguptaZiedinsMitra,galvin2011multistate} that 
the model on $\mathbb{T}_{\Delta}$  exhibits phase coexistence for
all sufficiently high $\lambda$, and the  point of phase
transition was identified asymptotically, as $\Delta$ tends to
infinity.  
The original model and its recent generalizations are also motivated
by applications in the field of communications \cite{KellyHardCore,RamananSenguptaZiedinsMitra,LueRamZie06}, in addition to the original statistical physics motivation.

The phase transition property on the infinite tree is known to be 
related to the algorithmic question of computing the partition 
function (or normalizing constant) associated with a Gibbs measure on a finite graph. 
Although the latter computation problem falls into the so-called
\#P-complete algorithmic complexity class for many models  (including
the standard hardcore model), there exist 
polynomial time approximation
algorithms, at least for certain models and corresponding ranges of
parameters.  More precisely, when the underlying parameters
are such that the corresponding Gibbs measure is  unique, 
 a polynomial time approximate computation of the corresponding
 partition function has been  shown to be possible for several discrete models  including 
the hardcore model  ~\cite{weitzCounting,BandyopadhyayGamarnikCounting},
matching
model~\cite{JerrumSinclairHochbaumApproxAlgorithms,BayatiGamarnikKatzNairTetali},
coloring model~\cite{BandyopadhyayGamarnikCounting,GamarnikKatz}, 
and some general binary models (models with two spin values) \cite{li2013correlation}. 
For some models, including  the standard hardcore and matching models, approximate
computation has been shown to be feasible whenever the model is in the
uniqueness regime.  For some other problems, including
counting the number of proper colorings of a graph, an 
approximation algorithm has been constructed only for a restricted
parameter range, although
it is conjectured to exist whenever the model is in the uniqueness regime. 
Furthermore, the converse has also been established for the hardcore
model and some of its extensions. Specifically, it
was shown in \cite{sly2010computational} that for certain  parameter values for
which there are multiple Gibbs measures,  approximate 
computation of the partition function in polynomial time
becomes impossible, unless P=NP. 
This link between the phase transition property on the infinite
regular tree $\T_{\Delta}$ and  hardness of
approximate compution of the partition function on a graph with
maximum degree $\Delta+1$ is conjectured to exist for  general models.

\subsection{Discussion of Results. } 
\label{subs-discussion}

In light of the connection between phase transitions and hardness of
computation mentioned above, an interesting problem to consider
is the problem of computing the volume of a (bounded) convex polytope,
obtained as the intersection of finitely many half-spaces.  
It is known that, while this volume computation  problem is \#P-hard \cite{dyer1988complexity},
it  admits a randomized polynomial time approximation scheme~\cite{DyerFriezeKannan}, regardless of the parameters of the model. In fact, such an algorithm
exists for computing the volume of an arbitrary convex body, subject to minor regularity conditions.
This motivates the  investigation of 
this problem from the phase transition perspective, by considering a
model in which the partition function is simply the volume of a
polytope.  
Towards this goal, 
 we introduce the {\em continuous hardcore model} on
a finite graph $\G$, which defines a measure that 
is supported on  the following  special type of polytope
\begin{align}
\label{polytope}
\Pol (\G)=\{\mbx=(x_u, u\in V): x_u \geq 0, x_u+x_v\le 1,
\forall~u \in V, (u,v)\in
E\}, 
\end{align}
where $V$ and $E$ are, respectively, the vertex and edge set of the graph $\G$.  
$\Pol(\G)$ is the linear programming relaxation of the independent set
polytope of the graph, and we refer to it as the linear
programming (LP) polytope of the graph
$\G$.   The continuous hardcore model with parameter $\lambda = 1$ is
simply the uniform measure on $\Pol (\G)$, and   
the associated partition function  is equal 
to the volume of the convex polytope $\Pol (\G)$.  

As in the
discrete case,  the continuous hardcore  model
defines a one-parameter  family of probability measures,  indexed by
the activity $\lambda > 0$  (see  Section \ref{section:model} for a precise
definition). We consider this model on an infinite regular tree
$\T_\Delta$.  
Our main result (Theorem \ref{theorem:MainResult})   
 is that, unlike the standard hardcore model, 
the continuous hardcore model on an infinite regular tree \emph{never}
exhibits a phase transition.   Namely, for every choice of $\Delta$
and $\lambda$, there is a  unique Gibbs measure for the continuous
hardcore model on  $\T_\Delta$ with activity $\lambda$. 
This result provides support for the conjecture that the link between the phase
transition property and hardness of approximate computation of  the partition function
is indeed valid for general models, including those in which the spin
values, or states of vertices, are continuous, rather than discrete. 
Moreover, in  Theorem \ref{th:DiffEq}  we characterize the cumulative
distribution function of the marginal at any node  of the continuous hardcore Gibbs measure (with
parameter $\lambda =1$) on the
infinite regular tree as the unique solution to a certain
ordinary differential
equation (ODE). 
An analogous  result is conjectured to hold for general $\lambda > 0$
 (see Conjecture \ref{conjecture:DiffEq}). 

We extend our result further by considering a natural interpolation
between the standard two-state hardcore and the continuous hardcore
models when $\lambda = 1$. 
Here, in addition to the hard-core constraint, 
the spin values $x_u$ are further restricted to belong to $[0,\epsilon]\cup
[1-\epsilon,1]$ for some fixed parameter $\epsilon\in (0, 1/2)$.   In
a sense made precise in Section \ref{subs-discussion}, when $\epsilon \rightarrow
1/2$ one obtains  the continuous hardcore model and as $\epsilon
\rightarrow 0$, it more closely 
resembles  the two-state hardcore model.    We establish, perhaps surprisingly, that the 
model  has a unique Gibbs measure for any positive value
$\epsilon>0$ (see Theorem \ref{theorem:MainResult2}), even when the 
discrete-hard core model (formally corresponding to $\ve = 0$) has
multiple Gibbs measures.   The same
argument does not easily  extend to  the case of general $\lambda$, 
and we leave this case open for further exploration. 

Our last result (Theorem \ref{theorem:MainResultRegularGraph}) concerns the computation of 
the volume of the LP  polytope of a regular  locally tree-like graph, in the limit as the
number of nodes and girth of the graph goes to infinity.   This result parallels some of the developments in
~\cite{BandyopadhyayGamarnikCounting}, where 
it is shown 
that the  partition functions associated with the standard hardcore model
defined on  a sequence of increasing regular locally tree-like graphs,  with growing girth, after appropriate
normalization, have a limit, and this limit coincides for all regular
locally tree-like graphs with degree $\Delta +1$ 
 when the model is in the uniqueness 
regime for the tree $\T_\Delta$, namely when $\lambda<\lambda_{c}(\Delta)$.  We establish a similar
result here, showing that the sequence of partition functions associated with the continuous
hardcore model on a sequence of increasing regular graphs with large girth,  after appropriate normalization,
has a well defined limit.   We establish a corresponding
approximation result for the continuous hardcore model, which is  valid for all $\lambda >
0$ since, as shown  in Theorem \ref{theorem:MainResult}, the
continuous hardcore  model has a unique Gibbs measure 
for every $\lambda$.  For the case $\lambda = 1$,  when combined with our 
characterization of the Gibbs measure in Theorem \ref{th:DiffEq}, this provides
a fairly explicit approximation of  the normalized volume of the LP 
polytope of  a  regular graph with large girth.

We now comment on the proof technique underlying our result.  
To establish uniqueness 
of the Gibbs measure, we establish the  correlation decay property. 
Unlike for the discrete hardcore model,  establishing correlation
decay  for continuous models is significantly more
challenging technically, 
since it involves analyzing recursive maps on the space of absolutely continuous
(density) functions, rather than one-dimensional or
finite-dimensional recursions, and  the function obtained as the
limit of these recursive maps is 
characterized as the solution to a certain nonlinear second-order ordinary differential
equation (ODE) with boundary conditions, 
rather than as the fixed point of a finite-dimensional map.   
The direct approach of establishing a contraction property, which is
commonly used in the analysis of discrete models,  appears unsuitable in 
our case.  
Instead, establishing existence, uniqueness and the correlation decay property 
entails the analysis of this ODE.
A key step that facilitates this analysis is the identification of a certain Hamiltonian
structure of the ODE.  This can be exploited, along with certain
monotonicity properties,  to establish uniqueness
of the Gibbs measure.   Characterization of the unique marginal
distribution at a node requires additional work, which is related to establishing uniqueness of the
solution to this ODE with suitable boundary conditions, and involves   
a detailed sensitivity analysis of a related parameterized
family of ODEs. 

\subsection{Outline of Paper and Common Notation} 
\label{subs-nottion}

The remainder of the paper is organized as follows. In Section
\ref{section:model} we  precisely define  the continuous hardcore
model, and a family of related models. 
Then, in Section \ref{sec-mainres}  we state our main results. 
 In Section~\ref{section:ProofHardCore} we prove our main
results, Theorem \ref{theorem:MainResult} and
\ref{theorem:MainResult2}, on correlation decay (and hence uniqueness
of the Gibbs measure) for the continuous
hardcore model and  its $\ve$-interpolations for $\ve \in (0,1/2]$. 
  In Section~\ref{eq:DE-Uniqueness}, we characterize the 
  marginal 
distribution of the unique Gibbs measure for the continuous hardcore model
with $\lambda = 1$ as the unique solution  to a certain 
nonlinear ODE.   The conjectured characterization for $\lambda \neq 1$
is described in Section \ref{subs-conjecture}.  
In Section~\ref{section:RegularGraphs} we prove our result regarding
the volume of the LP polytope of a regular graph with  large girth. 

 In what follows, given a set $A$, we let $\ind_{A}$ denote the indicator
function of the set $A$: $\ind_A (x) = 1$ if $x \in A$ and
$\ind_A(x) = 0$, otherwise, and when $A$ is finite,  let $|A|$ denote its cardinality.  For $a\in [0,1]$, let $\delta_a$
denote the Dirac delta measure at $a$,  let  $dx$ denote one-dimensional
Lebesgue measure, and given $\mbx = \{x_u, u \in V\}$, let $d \mbx$
denote $|V|$-dimensional Lebesgue measure.    Also, for any subset
$A\subset V$,   let $\mbx_A$ represent the vector
$(x_u, u \in A)$.  Let $\R$ and $\R_+$ denote the sets of real and non-negative real numbers,
respectively.  Given  any subset $S$
of $J$-dimensional Euclidean space $\mathbb{R}^J$, let ${\mathcal
  B}(S)$ represent the collection of Borel subsets of $S$.  For
conciseness, given a measure $\mu$ on ${\mathcal B} (\mathbb{R})$, for intervals $[a,b]$, we will
use ${\mathcal B}[a,b]$ and $\mu [a,b]$ to represent ${\mathcal
  B}([a,b])$ and $\mu([a,b])$, respectively.

\section{A family of hardcore models}\label{section:model}

Let $\G$ be a simple  undirected graph with  finite node set $V=V(\G)$
and  edge set $E = E(\G)$, and recall the associated LP 
polytope defined in \eqref{polytope}. 
We now introduce the continuous hardcore
model on the finite graph $\G$ associated with any parameter $\lambda  > 0$.   In fact, we will
 introduce a more general family of hardcore models that will 
  include both the discrete and continuous hardcore models in a
  common framework,  and allow us to also interpolate between the two. 
 Any model in this family is specified by a finite Borel measure $\fmeas$ on $[0,1]$, which we refer to as the
``free spin measure'' for the model. 
The free spin measure $\fmeas$ 
represents the weights the model puts on different states or spin values when
the graph $\G$ is a single isolated vertex; specific examples are
provided below. 
Given a free spin measure $\fmeas$ on the Borel sets ${\mathcal
  B}[0,1]$ of $[0,1]$ 
and $k \in \N$, let $\mu^{\otimes_k}$ represent the product measure on 
${\mathcal B}([0,1]^{k})$ with identical marginals equal to  $\fmeas$.

\begin{defn}
 The hardcore model corresponding to the graph $\G = (V,E)$ and free
 spin measure $\fmeas$  is  the
 probability measure $\hcmeas  = \hcmeas_{\G,\fmeas}$ 
 given by 
\begin{align}
\label{def-hcmeas}
   \hcmeas (A) \mean  \frac{1}{Z} \prodmu (A),  \qquad A \in {\mathcal  B} (\Pol), 
\end{align}
where $\Pol = \Pol (\G)$ is the  LP polytope defined in \eqref{polytope} 
and $Z$ is the partition function or  normalization constant given
by 
\begin{align}
\label{def-hcZ}
Z \mean  \prodmu  (\Pol). 
\end{align}
\end{defn}

The measure $\hcmeas$  is well defined as long as $Z  > 0$.   Since the hypercube $\{\mbx: 0\le x_u\le
1/2,  \forall u \in V\}$ is a subset of $\Pol(\G)$ for
every graph $\G$,  a simple sufficient condition for this to hold is that the
free spin measure satisfies $\fmeas[0,1/2] > 0$.   This will be the case in
all the models we study.

  We now describe the free spin measure associated with specific
  models. 
For $\lambda > 0$,   the free spin
measure of the two-state hard-core model with activity $\lambda$ is given by  $\mu  =
\dmeasl$, 
\begin{align}
\label{freespin-dhc}
\dmeasl (B)  \mean  \lambda \delta_1 (B) + \delta_0 (B), \qquad B \in
{\mathcal B} [0,1]. 
\end{align}
The measure $\dmeasl$ in \eqref{freespin-dhc} is discrete, supported on $\{0,1\}$ and
gives weights  $\lambda$ and $1$ to the values $1$ and $0$,
respectively,  and the corresponding $\P_{\G,\dmeasl}$ defines  the
standard (discrete) hardcore model with
parameter $\lambda > 0$.   This model was generalized to an 
$(M+1)$-state  hardcore model, for some
integer $M \geq 1$, in 
\cite{RamananSenguptaZiedinsMitra,galvin2011multistate}. 
Given a parameter
$\lambda > 0$, the free spin measure associated with a rescaled 
version of the latter model (that has support $[0,1]$) is 
\begin{align}
\label{freespin-multihc}
 \dmeaslM (B) = \sum_{i=0}^M \lambda^i \delta_{i\over M} (B), \qquad B \in
{\mathcal B}[0,1].  
\end{align}
The case $M = 1$ then recovers the standard (two-state) hardcore model.

We now define the continuous hardcore model on $\G$ with parameter $\lambda
> 0$ to be the measure $\hcmeas_{\G}^\lambda \mean
\hcmeas_{\G,\meas_\lambda}$ where the free spin measure $\meas_\lambda$
takes the form 
\begin{align}
 \label{freespin-chc}
 \meas_\lambda (B) \mean   \int_{B} \lambda^x  dx,  \quad B \in {\mathcal
   B}[0,1].  
\end{align}
Despite the similarity in the definitions in \eqref{freespin-chc} and
\eqref{freespin-multihc}, 
 an important difference is that  while $\dmeasl$ 
is  discrete, 
$\meas_\lambda$ in \eqref{freespin-chc} is absolutely continuous with respect to Lebesgue
measure.    In fact,  when $\lambda = 1$, the free spin measure is just the 
uniform distribution on $[0,1]$,  the corresponding Gibbs measure  is simply the uniform
measure on the polytope $\Pol$,  and $Z$ is the volume of the polytope
$\Pol$, 
as  already mentioned in Section
\ref{subs-discussion}. 
For  each activity parameter $\lambda  > 0$, 
we  also introduce a family of models, indexed by $\ve \in
(0,1/2)$ which we refer to as the 
$\ve$-continuous hardcore model 
that interpolate between the discrete and continuous hardcore models
with the same activity parameter. 
For $\lambda > 0$ and $ \ve  \in (0,1/2)$,  
the free spin measure of the $\ve$-continuous hardcore model with
activity 
parameter 
$\lambda$  is given by 
\begin{equation}
\label{freespin-epschc}
\meas^\ve_\lambda (B) \mean \frac{1}{2 \ve} \int_{B}\lambda^x \left(\ind_{[0,\ve]} (x)  + 
  \ind_{[1-\ve,1]} (x)\right) dx,  \qquad B \in {\mathcal B}[0,1]. 
\end{equation}
We now clarify the precise sense in which this interpolates 
between the discrete and continuous models.  
Given probability measures $\{\pi_{\ve}\}$ and $\pi$ on ${\mathcal
  B}[0,1]$, recall that $\pi_{\ve}$ is said to
converge weakly to $\pi$ as $\ve \rightarrow \ve_0$, 
 if for every bounded continuous function $f$ on
$[0,1]$, $\int_{[0,1]} f(x)
\pi_{\ve} (dx) \rightarrow \int_{[0,1]} f(x) \pi (dx)$ as $\ve
\rightarrow \ve_0$. 
For any $\lambda > 0$, when $\ve \uparrow 1/2$, $\meas^{\ve}_\lambda$
converges weakly to  $\meas_\lambda$, the free spin measure of  the
continuous hardcore model with parameter $\lambda > 0$, 
 as in \eqref{freespin-chc}, whereas as $\ve \downarrow 0$, 
$\meas^{\ve}_\lambda$ converges weakly to $\dmeasl$, the
corresponding free spin measure of 
the two-state hardcore model as in \eqref{freespin-dhc}. 

Given any hardcore model on a finite graph $\G$ with free spin
measure $\fmeas$,  we  let $\mbX = (X_u, u\in V)$ 
denote a random element distributed according to 
  $\P_{\G,\fmeas}$, and refer to  $X_u$
as the spin value at $u$.   Recall that given a subset $S$  of nodes in $V(\G)$,
we use $\mbX_S = (X_u, u \in S)$ to denote the natural projection of $\mbX$ to the coordinates corresponding to $S$.
The constructed hardcore probability distributions $\P_{\G, \mu}$ are 
 Markov random fields, or Gibbs measures, 
in the sense that they satisfy the following spatial Markov property. Given any subset $S\subset V$, let $\partial S$ denote
the set of nodes $u$ in  $V\setminus S$ that have neighbors in $S$,
that is, for which $(u,v) \in E$ for some $v \in S$.  Then
for  every vector 
$\mbx=(x_u, u\in V)\in \Pol (\G)$ 
that lies in the support of $\P = \P_{\mu, \G}$, we have 
\begin{align*}
\P(\mbx_S |\mbx_{V\setminus S})=\P(\mbx_S|\mbx_{\partial S}). 
\end{align*}
Namely, the joint probability distribution of spin values $X_u$
associated with nodes $u\in S$, conditioned on the spin values at all
other nodes  of the graph is equal to the joint distribution obtained on just conditioning on spin
values at the boundary of $S$.  
 Of course, such a conditioning should be well defined, which is
 easily seen to be the case for the hardcore models we consider.

\section{Main Results}
\label{sec-mainres}

We now turn to the setup related to the main results in the paper.
We first recall some standard graph-theoretic notation. 
For every node
$u\in V$, $\calN(u)=\calN_\G(u)$ denotes the set of neighbors of $u$, namely the set $\{v: (u,v)\in E\}$.
The cardinality of $\calN(u)$ is called the degree of the node $u$ and is denoted by $\Delta(u)$. A leaf is a node with degree $1$.
Given a positive integer $\Delta$,  a graph is called $\Delta$-regular
if $\Delta(u)=\Delta$ for all nodes of the graph. The graph theoretic
distance between nodes $u$ and $v$ is the length of a shortest path
from $u$ to $v$ measured in terms of the number of edges on the path.
Namely, it is the smallest $m$ such that there exist nodes $u_0=u,u_1,\ldots,u_m=v$ such that $(u_i,u_{i+1}), i=0,1,\ldots,m-1,$ are edges.
A cycle is a path $u_0=u,u_1,\ldots,u_m$ such that $m\ge 3,~u_m=u_0$, and all $u_1,\ldots,u_m$ are distinct.
The girth $g=g(\G)$ of the graph $\G$ is the length of a shortest
cycle.

 Let $\T_{n,\Delta}$ denote a  rooted regular tree with
 degree $\Delta+1$ and depth $n$, which is a finite tree with a special
 vertex called the root node, in which every node has degree $\Delta+1$ except for the 
root node, and the leaves, which is 
 the collection of nodes that are at a graph-theoretic distance $n$
from the root node and denoted $\leaves$.    Each leaf has degree $1$, 
and the root has degree $\Delta$. 
Note that $\leaves$ is also  the
boundary of the remaining nodes of $\T_{n,\Delta}$ (which we refer to as internal
nodes).   Fix $\lambda > 0$  and let $\P_{n,\Delta,\lambda}$ 
represent the
(continuous) hardcore 
distribution on $\T_{n,\Delta}$ with parameter $\lambda > 0$, 
corresponding to the free spin   measure $\meas_\lambda$  in
\eqref{freespin-chc}.  
We denote the (cumulative) distribution
function of the marginal of $\P_{n, \Delta, \lambda}$ at the root node 
 by $F_{n,\Delta,\lambda}(\cdot)$.   Clearly, $F_{n,\Delta,\lambda}(\cdot)$ is
 absolutely continuous and we denote its density by
 $f_{n,\Delta,\lambda}(\cdot)$.  
Given an arbitrary realization of spin values at the boundary
$\mbx_{\leaves}$, we also
let $F_{n,\Delta,\lambda}(\cdot|\mbx_{\partial\T_{\Delta,n}})$ denote
the cumulative distribution function of the conditional distribution
of $\P_{n, \Delta, \lambda}$ at the root given $\mbx_{\leaves}$.  
It can be shown (see \eqref{eq:FnIteration} with $\mu = \meas_\lambda$) 
that for  $n \geq 2$, 
$F_{n,\Delta,\lambda}(\cdot|\mbx_{\partial\T_{\Delta,n}})$ has 
a density, which we denote by 
$f_{n,\Delta,\lambda}(x|\mbx_{\leaves})$. 
In particular,
for $x \in [0,1]$, 
\[ F_{n,\Delta,\lambda}(x)=\int_0^x f_{n,\Delta,\lambda}(t)dt, \qquad 
F_{n,\Delta,\lambda}(x|\mbx_{\partial\T_{\Delta,n}})=\int_0^x
f_{n,\Delta,\lambda}(t|\mbx_{\partial\T_{\Delta,n}})dt. 
\]

We now state our first main result, which is proved in
Section~\ref{subs-pf}. 
 For any absolutely continuous function $F$,  we let $\dot{F}$ denote the
derivative of $F$, which exists almost everywhere. 
 Also, for any real-valued function $g$ on $[0,\infty)$ and
compact set ${\mathcal K} \subset [0,\infty)$, we let
$||g(\cdot)||_{{\mathcal K}} \mean \sup_{x \in {\mathcal K}}
|g(x)|$. 
 
\begin{theorem}\label{theorem:MainResult}
For every $\Delta\ge 1$ and  $\lambda>0$, there exists a
non-decreasing function
$F_{\Delta, \lambda}$ with $F_{\Delta, \lambda}(0) = 0$ that is continuously differentiable on
$(0,\infty)$, and satisfies, for any compact subset ${\mathcal K} \subset [0,1]$,
\begin{align}\label{eq:uniqueness}
&\lim_{n\rightarrow\infty}\sup\Big|\Big|F_{n,\Delta,\lambda}(\cdot|\mbx_{\partial\T_{\Delta,n}})-F_{\Delta,\lambda}(\cdot) \Big|\Big|_{[0,1]} = 0, \\
\label{eq:uniqueness2} 
&\lim_{n\rightarrow\infty}\sup\Big|\Big|\dot{F}_{n,\Delta,\lambda}(\cdot|\mbx_{\partial\T_{\Delta,n}})-
\dot{F}_{\Delta,\lambda}(\cdot) \Big|\Big|_{{\mathcal K}}=0,
\end{align}
  where the supremum is over all boundary conditions
$\mbx_{\partial\T_{\Delta,n}} \in [0,1]^{|\partial\T_{\Delta,n}|}$. 
\end{theorem}

\begin{remark}
\label{rem-mainres}
{\em  The relation  (\ref{eq:uniqueness}) of 
 Theorem~\ref{theorem:MainResult}
implies that the  cumulative distribution function of the marginal 
distribution at the root is asymptotically independent from the
boundary condition.  
In particular, the model exhibits the correlation decay property
regardless of the values of $\Delta$ and $\lambda$ (which implies no
phase transition).  
 In fact, it follows from
Theorem~\ref{theorem:MainResult}  
that there exists a unique Gibbs
measure on the infinite $(\Delta+1)$-regular tree  
and that this measure is translation invariant and its marginal distribution function at any
node  is equal to $F_{\Delta, \lambda}$. 
Relation (\ref{eq:uniqueness2}) shows that the decay of correlations property extends to the marginal density.
}
\end{remark}

Next, we  provide a more explicit characterization of the marginal 
distribution function $F_{\Delta, \lambda}$ in the special case
$\lambda = 1$, 
which is the quantity of interest for computing the
volume of the polytope $\Pol (\G)$.  We show that this  limit is the unique
solution to a certain  first-order ODE.

\begin{theorem}\label{th:DiffEq}
For $\lambda = 1$ and $\Delta \geq 1$, there exists a unique $C =C_{\Delta, 1} > 0$ such that   the 
ODE 
\begin{equation}
\label{ODE}
  \dot{F}(z) =  C ( 1 - F^{\Delta +1} (z))^{\Delta/(\Delta+1)},
  \quad z 
  \in (0,\infty), 
\end{equation}
with boundary conditions 
\begin{equation}
\label{bc}
  F (0) = 0  \quad \mbox{ and } \quad \inf \{t > 0:  F(t) = 1 \} = 1, 
\end{equation}
has a solution.  Moreover, the ODE {\eqref{ODE}-\eqref{bc}} with $C =
C_{\Delta,1}$ has a unique solution $\fstar_{\Delta,1}$.
Furthermore, $\fstar_{\Delta,1}=F_{\Delta, 1}$, where
$F_{\Delta,1}$ is the limit distribution function of Theorem
\ref{theorem:MainResult}. 
\end{theorem}

The proof of  Theorem  \ref{th:DiffEq} is given in   Section
\ref{eq:DE-Uniqueness}.  
In fact, we believe a
generalization is possible to all $\lambda > 0$. 
 Specifically, as stated in  Conjecture
\ref{conjecture:DiffEq} at the end of Section \ref{eq:DE-Uniqueness},  we
believe $F_{\Delta, \lambda}$ also admits a characterization in terms of 
a differential equation, although a more complicated second-order
non-linear differential equation, but we defer the validation of 
such a conjecture to future work.

The  behavior  of the continuous hardcore model described above should be contrasted with
that of the discrete hardcore model 
for which, as discussed in the introduction, the phase transition
point on a $(\Delta+1)$-regular tree is
$\lambda_c=\Delta^\Delta/(\Delta-1)^{\Delta+1}$.   In particular, when
$\Delta \geq 5$, $\lambda_c< 1$ and so the discrete hardcore model on the tree with $\lambda
= 1$ admits multiple Gibbs measures.  
This raises the natural question as to what happens for the 
$\ve$-interpolated model, with free spin measure $\nu^\ve_1$, as in
\eqref{freespin-epschc}.   It is natural
to expect that this model would 
 behave just like the standard hardcore model with $\lambda = 1$ 
for sufficiently small $\ve$.   Somewhat
surprisingly, we show that this is not the case.    By establishing a
 correlation decay
property similar to that described in Remark \ref{rem-mainres}, 
in Theorem \ref{theorem:MainResult2} we show that there is a unique Gibbs measure
for the $\ve$-interpolated model for every positive $\ve$, no matter
how small.

\begin{theorem}\label{theorem:MainResult2} 
For every $\Delta\ge 1$,  and $\ve \in (0,1/2)$, let 
$F_{n,\Delta}^{(\ve)}$ denote the cumulative distribution of the marginal
of the Gibbs measure $\P_{\mathbb{T}_{\Delta,n}, \nu^\ve_1}$ 
at the root
of $\mathbb{T}_{\Delta,n}$.   Then 
there exists a non-decreasing continuous function 
$F_{\Delta}^{(\ve)}$ with $F_{\Delta}^{(\ve)}(z) = 0$ for $z \leq 0$, $F_{\Delta}^{(\ve)}(z) = 1$ for $z \geq 1$, that satisfies 
\begin{align}\label{eq:uniqueness-eps}
&\lim_{n\rightarrow\infty}\sup\Big|\Big|F_{n,\Delta}^{(\ve)}(\cdot|\mbx_{\partial\T_{\Delta,n}})-F_{\Delta}^{(\ve)}(\cdot)
\Big|\Big|_{[0,1]} = 0, 
\end{align}
  where the supremum is over all boundary conditions
$\mbx_{\partial\T_{\Delta,n}} \in [0,1]^{|\partial\T_{\Delta,n}|}$. 
\end{theorem}

We now turn to the implications of our results for volume computation.
Specifically, applying Theorem~\ref{theorem:MainResult},  
we are able to compute
asymptotically the volume of the  LP polytope associated
with any  regular graph that is locally tree-like (that is, with large
girth). The proof of Theorem \ref{theorem:MainResultRegularGraph} is given in
Section  \ref{subs-pf-Regulargraph}.

\begin{theorem}\label{theorem:MainResultRegularGraph}
Fix $\lambda > 0$ and $\Delta \geq 1$, and let $F_{\Delta, \lambda}$
be as in Theorem \ref{theorem:MainResult}. 
Let $\G_n, n\ge 1,$ be a sequence of $\Delta$-regular graphs with
$g(\G_n)\rightarrow\infty$,  
and let $Z_{\G_n,\lambda}$ be the associated partition function as 
defined by (\ref{def-hcZ}) with $\mu = \mu_\lambda$ in
\eqref{def-hcmeas} and $\Pol = \Pol (\G_n)$.  Then 
\begin{align}
\lim_{n\rightarrow\infty}{\ln Z_{\G_n,\lambda}\over |V(\G_n)|}
&=
-\ln\int_{0\le x\le 1}\lambda^x F^\Delta_{\Delta-1,\lambda}(1-x)dx \notag \\
&\qquad -{\Delta\over 2}\ln\int_{0\le x\le 1}\dot F_{\Delta-1,\lambda}(x)F_{\Delta-1,\lambda}(1-x)dx. \label{eq:VolumeLimit}
\end{align}
\end{theorem}

Combining Theorem \ref{theorem:MainResultRegularGraph} with
Theorem \ref{th:DiffEq}, we see that in the special case $\lambda=1$, the volume $Z_{\G_n,1}$ of the polytope
$\mathcal{P}(\G_n)$ satisfies 
\begin{align*}
\lim_{n\rightarrow\infty}{\ln Z_{\G_n,1}\over |V(\G_n)|}=\gamma,
\end{align*}
where $\gamma$, 
 which stands for the right-hand side of (\ref{eq:VolumeLimit}) with
 $\lambda = 1$, 
takes the form 
\begin{align*} 
\gamma &= - \ln \int_0^1  \fstar_{\Delta-1, 1}^\Delta (1-x) dx \\
& \qquad - \frac{\Delta}{2} \ln \int_0^1 \left( 1  -
  \fstar_{\Delta  -1, 1}^{\Delta  -1} (x) \right)^{\frac{\Delta-1}{\Delta}} 
\fstar_{\Delta-1,  1} (1-x) \, dx, 
\end{align*}
where $\fstar_{\Delta-1,1}$ is the unique solution to
\eqref{ODE}-\eqref{bc} with $C = C_{\Delta,1}$, as identified in Theorem \ref{th:DiffEq}

 This result provides a fairly explicit expression for the exponential limit of the volume of such a  polytope, via 
the solution $\fstar_{\Delta-1,  1}$ of the ODE,  which can be computed, for example, numerically.
A  similar expression for general $\lambda$ would be obtained if  Conjecture~\ref{conjecture:DiffEq}
were shown to be valid.

\section{Analysis of Continuous hardcore Models}\label{section:ProofHardCore}
For ease of exposition, we fix $\Delta \geq 1$ and for each $n \geq 1$, use the notations $\T_n$ and $\Pol_n$ in place
of $\T_{n,\Delta}$ and $\Pol (\T_{n,\Delta})$, respectively.  
Also,  in order to present a unified 
proof of Theorems \ref{theorem:MainResult} and
\ref{theorem:MainResult2} to the extent possible, we will first fix 
any spin measure $\fmeas$ that is absolutely continuous with
respect to Lebesgue measure, let $\rn$ denote its density, and let 
$F_n$ and  $f_n$, respectively, 
denote the cumulative distribution function and density of the marginal at the root node of
the hardcore model on $\T_n$ with free spin measure $\fmeas$.  Also,
in analogy with the definitions in Section \ref{sec-mainres}, let 
$F_n (\cdot|\mbx_{\partial \T_n})$ and,  for $n \geq
  2$, $f_n(\cdot|\mbx_{\partial \T_n})$, denote the corresponding
  conditional distribution functions and density given the boundary
  condition $\mbx_{\partial \T_n} \in [0,1]^{|\partial \T_n|}$. 
Also,  let  $Z_n$ denote the corresponding hardcore partition function
\eqref{def-hcZ}. 

 The proof of Theorem
\ref{theorem:MainResult} entails several steps. First, in Section \ref{subs-monot} we 
 establish a monotonicity result, which allows one to only consider 
the  cases when the boundary condition $\mbx_{\partial\T_{n}}$ is the
vector of  zeros or  is the vector of ones.  
 Then in Sections 
\ref{subs-densityevol} and  \ref{subs-denslimits} we derive iterative
formulas for $F_{2n}$ and $F_{2n+1}$ and show that each of these
sequences is pointwise monotonic in $n$, and thus converge to limiting
functions $F_e$ and $F_o$, respectively.  
In Section \ref{subs-diffeq}, we characterize $F_e$ and $F_o$ in terms
of certain  ODEs, and also identify a certain Hamiltonian
structure that leads to an invariance property 
in the particular case of the continuous and
$\ve$-interpolated 
models.    Finally, in Section \ref{subs-pf}, we  use this invariance
property to prove   Theorems  \ref{theorem:MainResult}  and 
\ref{theorem:MainResult2}.

\subsection{Monotonicity property}
\label{subs-monot}

Given a  spin measure $\mu$, let $\mbzero_{\partial\T_{n}}$ and $\mbone_{\partial\T_{n}}$,
respectively, be the boundary condition 
corresponding to setting the values for the leaves
of $\T_n$ to be all zeros and all ones.   
In Lemma \ref{lemma:monotonicity} we state a monotonicity property 
for general models, having discrete or continuous free spin measure. 
This property  is well known for the special case of the standard (two-state) hardcore
model, and was further extended in~\cite[Lemma 2.2]{galvin2011multistate} 
to  the multi-state hardcore model with free-spin measure
$\meas_\lambda^{M+1}$   in \eqref{freespin-multihc} for any integer $M
\geq 1$.  
For completeness, the proof of Lemma \ref{lemma:monotonicity} is provided in 
Appendix \ref{ap-monotonicity}. 

\begin{lemma}\label{lemma:monotonicity}
For $n \geq 1$, every boundary condition $\mbx_{\partial\T_n}$ and every $\ex\in [0,1]$,
\begin{align*}
F_{n}(\ex|\mbzero_{\partial\T_{n}})\ge
F_{n}(\ex|\mbx_{\partial\T_n})\ge
F_{n}(\ex|\mbone_{\partial\T_{n}}), \\
 \end{align*}
 when $n$ is even and
 \begin{align*}
F_n(\ex|\mbzero_{\partial\T_{n}})\le  F_n(\ex|\mbx_{\partial\T_n})\le F_n(\ex|\mbone_{\partial\T_{n}}),
\end{align*}
when $n$ is odd.
\end{lemma}

\subsection{A Recursion for the Marginal Distribution Functions}
\label{subs-densityevol}

We now derive iterative formulas for the functions $F_n.$  
Let $\T_0$ denote the trivial tree consisting of an isolated vertex. Then  from
\eqref{freespin-chc},  $Z_0 \mean \mu [0,1]$, where $\mu$ is the free
spin measure, and 
the associated distribution function $F_0$ takes the form 
\begin{align}
 F_0(\ex)&= \frac{\mu[0,\ex]}{\mu[0,1]},  \quad 
 \ex\in [0,1].   \label{eq:F0}
\end{align}

\begin{lemma}\label{lemma:DensityIteration}
Given any free spin measure $\mu$, 
for every $n \geq 1$, $F_{n}(z) = 0$ for $z \leq 0$, $F_{n}(z) = 1$ for $z
\geq 1$, $F_n$ is nondecreasing on $(0,1)$ and  the following properties hold: 
\begin{enumerate}
\item 
For $z \in [0,1]$, 
$F_{n}(\ex|\mbzero_{\partial \T_n})=F_{n-1}(\ex)$ and
$F_{n+1}(\ex|\mbone_{\partial \T_n})=F_{n-1}(\ex)$. 
\item 
Moreover, 
 \begin{align}\label{eq:FnIteration}
F_n(\ex)&=\frac{Z_{n-1}^{\Delta}}{Z_n}\int_{[0,\ex]}
F_{n-1}^{\Delta}(1-x_{u_0})\mu(dx_{u_0}),
\quad \ex \in [0,1]. 
\end{align}
\item 
Furthermore, 
\begin{align}\label{eq:FnATone}
\int_{[0,1]}  F_{n-1}^\Delta (1-t)\mu(dt) ={Z_n\over Z_{n-1}^\Delta}, 
\end{align}
and 
\begin{align}
\label{lbd}
 \liminf_{n \rightarrow \infty} {Z^\Delta_{n-1}\over Z_{n}}  > 0. 
\end{align} 
\end{enumerate}
\end{lemma}
\begin{proof}
The values of $F_n$ on $(-\infty, 0]$ and $[1,\infty)$,  and the
monotonicity of $F_n$ follow
immediately from the fact that $F_n$ is the cumulative distribution
function of a random variable  with support in $[0,1]$.  
Next, given the boundary condition $\mbzero_{\partial \T_n}$, the
hard-core constraints $x_u+x_v\le 1$  for every leaf
node $u$ and its parent $v$,  reduces to 
the vacuous constraint $x_v\le 1$. 
Thus, the boundary condition $\mbzero_{\partial \T_n}$ translates to a free boundary (no boundary)
condition on the tree $\T_{n-1}$. Similarly, the boundary condition
$\mbone_{\partial \T_n}$ forces $x_v$ to be zero for every parent $v$ of a leaf of the
tree $\T_n$, which in turn translates into a free boundary condition for the tree $\T_{n-2}$. This proves the first assertion of the lemma.

We now establish the second part of the lemma.
Let $u_0$ denote the root of the tree
$\T_n$ and note that for every $n\ge 1$, letting $\mbx=(x_u, u\in
V(\T_n))$, we have for every $\ex \in [0,1]$,   
\begin{align}\label{eq:MarginalCDFroot}
F_n (\ex) & = \frac{1}{Z_n} (\mu^{\otimes_{|V(\T_n)|}}) \left\{ \mbx \in
  \mathcal{P}_n: x_{u_0} \le \ex \right\}.  
\end{align}
Now,  let $u_1,\ldots,u_{\Delta}$ 
denote the children of the root $u_0$. Each child $u_i$  is the root of a tree
$\T_{n-1}^i$ that is an isomorphic copy of $\T_{n-1}$. 
The constraint $\mbx\in \Pol_n$ translates into  the
constraints $x_{u_0}+x_{u_i}\le 1, i=1,2,\ldots,\Delta$, plus 
the condition that the natural restriction $\mbx_{\T_{n-1}^i}$ of $\mbx$
to the subtree $\T_{n-1}^i$ lies in $\mathcal{P}_{n-1}^i \mean \mathcal{P}(\T_{n-1}^i)$. 
 Since these subtrees are non-intersecting, we obtain
\begin{align}
& (\mu^{\otimes_{|V(\T_n)|}} ) \left\{ \mbx \in
  \mathcal{P}_n: x_{u_0} \le \ex \right\} \notag \\
\notag \\
&=\int_{0}^{\ex} d\mu (x_{u_0})  \prod_{1 \leq i \leq \Delta}
(\mu^{\otimes_{|V(T_{n-1}^i)|}} ) \left\{ \mbx
  \in \mathcal{P}_n^i: x_{u_i} \leq 1  - x_{u_0}\right\}.
\label{eq:recursion}
\end{align}
Now, for each $1 \leq i \leq \Delta$,  we recognize the identity
\begin{align*}
\frac{1}{Z_{n-1}}(\mu^{\otimes_{|V(T_{n-1}^i)|}} ) \left\{ \mbx
  \in \mathcal{P}_n^i: x_{u_i} \leq 1  - x_{u_0}\right\}
=F_{n-1}(1-x_{u_0}).
\end{align*}
Combined with \eqref{eq:recursion} and \eqref{eq:MarginalCDFroot},
this yields \eqref{eq:FnIteration}. 

Setting $F_n(1)  = 1$ in \eqref{eq:FnIteration}, we obtain \eqref{eq:FnATone}. 
Furthermore, since $F_{n-1}$ is bounded by $1$  and 
$\mu$ is a finite Borel measure, 
\eqref{eq:FnATone} implies that 
$\sup_{n} \frac{Z_n}{Z_{n-1}^\Delta} \leq  \mu [0,1] <
\infty$,  
which yields \eqref{lbd}. 
\end{proof}

Combining Lemma~\ref{lemma:monotonicity} and the first part of
Lemma~\ref{lemma:DensityIteration} we now obtain a different
monotonicity result  along certain subsequences. 

\begin{cor}\label{coro:Monotonicity}
For every free spin measure $\mu$, for $n\geq 1$ and $\ex\in [0,1]$, $F_{2n+1}(\ex)\le
F_{2n-1}(\ex)$ and $F_{2n}(\ex)\ge F_{2n-2}(\ex)$. 
Furthermore, for every $n_1,n_2\in\Z_+$, with $F_{2n_1+1}(\ex)\ge F_{2n_2}(\ex)$.
\end{cor}

\begin{proof}
Once again, let $u_0$ denote the root of the tree $\T_n$ and label its
children as  $u_1,\ldots,u_\Delta$. 
Consider the random vector $\mbX$ chosen according to the
 hardcore measure $\P = \P_{\T_n, \mu}$, and
let $\P_{\partial \T_n}$ denote the marginal of $\P$ on the leaves
$\partial \T_n$. 
Then $F_n$ is the cumulative distribution function of the marginal at the
root, that is,  $F_n(\ex)=\P(X_{u_0}\le \ex)$.  
Thus, for every odd $n \geq 3$, using Lemma~\ref{lemma:monotonicity}
for the inequality and Lemma~\ref{lemma:DensityIteration}(1) for the last
equality below, we have 
\begin{align*}
F_{n}(\ex) =\P(X_{u_0}\le \ex) 
&=\int_{\mbx_{\partial\T_n} \in [0,1]^{|\partial\T_n|}}
\P(X_{u_0}\le\ex|\mbX_{\partial\T_n} =\mbx_{\partial \T_n})\P_{\partial \T_n} (d\mbx_{\partial\T_n})\\
&\le \P(X_{u_0}\le \ex|\mbone_{\partial\T_n})\\
&=F_{n-2}(\ex). 
\end{align*} 
Similarly, for every even $n$
we obtain $F_n(\ex)\ge F_{n-2}(\ex)$ for every $\ex$. 
Finally, to establish
the last inequality suppose first that $n_1 \ge n_2$. Then since the
first assertion of the lemma implies $F_{2n_1}(\ex)\ge F_{2n_2}(\ex)$, by a
similar derivation, we have 
\begin{align*}
F_{2n_1+1}(\ex)
&=\int_{\mbx_{\partial\T_{2n_1+1}}}\P(X_{u_0}\le \ex|\mbX_{\partial\T_{2n_1+1}}=\mbx_{\partial\T_{2n_1+1}})\P_{\partial \T_{2n_1+1}}  (d\mbx_{\partial\T_{2n_1+1}}) \\
&\ge \P(X_{u_0}\le \ex|\mbzero_{\partial\T_{2n_1+1}})\\
&=F_{2n_1}(\ex)\\
&\ge F_{2n_2}(\ex). 
\end{align*}
Conversely, if $n_1<n_2$, then $2n_1 + 1 \leq 2n_2 - 1$ and we use instead 
\begin{align*}
F_{2n_2}(\ex)
&=\int_{\mbx_{\partial\T_{2n_2}}}\P(X_{u_0}\le
\ex|\mbX_{\partial\T_{2n_2}}=\mbx_{\partial\T_{2n_2}})\P_{\partial \T_{2n_2}} (d \mbx_{\partial\T_{2n_2}}) \\
&\le \P(X_{u_0}\le \ex|\mbzero_{\partial\T_{2n_2}})\\
&=F_{2n_2-1}(\ex)\\
&\le F_{2n_1+1}(\ex).
\end{align*} 
\end{proof}

\subsection{A Convergence Result}
\label{subs-denslimits}

The monotonicity result of Corollary \ref{coro:Monotonicity} allows us
to argue the existence of the following 
pointwise limits: for $\ex \in [0,1]$, 
\begin{align}
\label{FoFe}
F_o(\ex) \mean \lim_{n\rightarrow\infty}F_{2n+1}(\ex), \qquad 
F_e(\ex) \mean \lim_{n\rightarrow\infty}F_{2n}(\ex). 
\end{align}
Also, 
note that by Corollary
\ref{coro:Monotonicity}, for $\ex \in [0,1]$, 
\begin{align}\label{eq:dominance}
1\ge F_{2n+1}(\ex)\ge F_{o}(\ex)\ge F_{e}(\ex)\ge F_{2n}(\ex)\ge 0. 
\end{align} 

Clearly $F_o$ and $F_e$ are measurable and bounded. So,  
we can  define 
\begin{align}\label{eq:C_o}
C_o\mean \left(\int_{[0,1]} F_o^\Delta(1-t)\mu( dt) \right)^{-1},
\end{align}
and
\begin{align}\label{eq:C_e}
C_e \mean \left(\int_{[0,1]}F_e^\Delta(1-t) \mu(dt) \right)^{-1}. 
\end{align}
Note that by \eqref{FoFe}, since $\mu$ is a finite Borel measure, the
dominated 
convergence theorem, \eqref{eq:FnATone}  and \eqref{lbd} imply 
\begin{align}
\label{finiteconst}
C_o &= \lim_{n \rightarrow \infty}
\left(\int_{[0,1]} F_{2n+1}^\Delta(1-t) \mu (dt) \right)^{-1}  =
\lim_{n \rightarrow \infty} \frac{Z_{2n+1}^\Delta}{Z_{2n+2}} > 0. 
\end{align}
Moreover,  by \eqref{eq:dominance},  the dominated convergence theorem 
 and \eqref{eq:FnATone}, we have 
\begin{align}
 \label{finiteconst2}
C_e &=  \lim_{n \rightarrow \infty}\left(\int_{[0,1]}
  F_{2n}^\Delta(1-t) \mu(dt) \right)^{-1}    =
\lim_{n \rightarrow \infty} \frac{Z_{2n}^\Delta}{Z_{2n+1}}. 
\end{align} 
The first equality above, together with \eqref{eq:dominance} and \eqref{eq:F0}, also show that 
\begin{align}
\label{finiteconst3}
C_e^{-1} & \geq \frac{1}{(\mu[0,1])^\Delta} \int_{[0,1]}(\mu[0,1-t])^{\Delta}
\mu(dt) \geq
\frac{(\mu[0,1/2])^{\Delta+1}}{(\mu[0,1])^\Delta}. 
\end{align}

We now derive an analogue of (\ref{eq:FnIteration}) for the limits $F_o$
and $F_e$, and strengthen the convergence in \eqref{FoFe}. 

\begin{cor}\label{cor:Limits} Suppose the free spin measure $\mu$
  satisfies 
$\mu[0,1/2] > 0$. Then $C_e, C_o \in (0,\infty)$, 
$F_o(0) = F_e(0) = 0$, $F_o(1) = F_e(1) = 1,$ 
and  for $\ex \in [0,1]$, 
\begin{align}
\label{relationFo}
F_o(\ex)&=C_e \int_0^\ex  F_e^\Delta(1-t) \mu (dt), \\
\label{relationFe}
F_e(\ex)&=C_o\int_0^\ex  F_o^\Delta(1-t) \mu(dt).  
\end{align}
Moreover, we also have 
\begin{align*}
\lim_{n \rightarrow \infty}||F_{2n+1}(\cdot)-F_{o}(\cdot)||_{[0,1]}&=0, \\
\lim_{n \rightarrow \infty}||F_{2n}(\cdot)-F_{e}(\cdot)||_{[0,1]}&=0. 
\end{align*}
Finally, suppose $\mu$ has density $\rn$ and  that ${\mathcal I}$ is
an  open set of continuity points of $\rn$. 
Then $\rn$, $F_n$, $F_o$ and $F_e$ are continuously differentiable on
${\mathcal I}$, and for every compact subset ${\mathcal K} \subset
{\mathcal I}$, 
\begin{align*}
\lim_{n \rightarrow \infty}||\dot F_{2n+1}(\cdot)-\dot F_{o}(\cdot)||_{{\mathcal K}}&=0, \\
\lim_{n \rightarrow \infty}||\dot F_{2n}(\cdot)-\dot F_{e}(\cdot)||_{{\mathcal K}}&=0.
\end{align*}
\end{cor}
\begin{proof}
The values of $F_o$ and $F_e$ at $0$ and $1$ follow directly from the
corresponding values of $F_n$ from Lemma \ref{lemma:DensityIteration}
and \eqref{FoFe}. 
Since \eqref{eq:dominance} implies $C_o \leq C_e$ the 
 estimates \eqref{finiteconst} and \eqref{finiteconst3}  imply  that as long as $\mu[0,1/2] > 0$, both $C_o$
 and $C_e$ lie in $(0,\infty)$.  
For $\ex \in [0,1]$, let  $F^*_o(\ex)$ and $F^*_e(\ex)$
equal the right-hand sides of \eqref{relationFo} and 
\eqref{relationFe}, 
respectively.  
  Taking limits on both sides of \eqref{eq:FnIteration} along
 odd $n$, and using \eqref{FoFe}, \eqref{finiteconst2} and the
     dominated convergence theorem, we obtain \eqref{relationFo}. 
The relation \eqref{relationFe} is obtained analogously, using
\eqref{finiteconst} instead of \eqref{finiteconst2}.  
The latter relations show that $F_o$ and $F_e$ are continuous. Since
they are also  pointwise
monotone limits of the sequences $F_{2n+1}$ and $F_{2n}$, respectively (see
Corollary \ref{coro:Monotonicity} and \eqref{FoFe}), by Dini's theorem, the convergence
is in fact uniform.   

We now prove the last property of the lemma, even though we do not use
it in the sequel. 
Suppose $\mu$ has  density $\rn$ that is continuous on
${\mathcal  I}$.  Then, \eqref{eq:F0} and \eqref{eq:FnIteration} show that
for every $n$, $F_n$ is absolutely continuous and $\dot{F}_n(z) =
Z_n^{-1}Z_{n-1}^\Delta F_{n-1}^\Delta (1-z) \rn (z)$,  from which it
follows that $\dot{F}_n$ is continuous on ${\mathcal I}$.  
Likewise, the continuous differentiability of 
$F_o$ and $F_e$ on ${\mathcal I}$ can  be deduced from
 (\ref{relationFo}) and (\ref{relationFe}).   The uniform convergence of the derivatives on
any compact subset ${\mathcal K} \subset {\mathcal I}$ is a direct
consequence of \eqref{eq:FnIteration}, \eqref{relationFo}-\eqref{relationFe} and the 
uniform convergence of $\{F_{2n+1}\}$ to $F_o$ and $\{F_{2n}\}$ to
$F_e$. 
\end{proof}

\begin{remark}
\label{rem-whatsleft}
{\em  We now claim (and justify below) that to prove the correlation
  decay property in Theorems~\ref{theorem:MainResult} and
  \ref{theorem:MainResult2}, it suffices to show (for the respective
  models) that $C_e = C_o$.  Indeed, by Lemma \ref{lemma:monotonicity},
  Lemma \ref{lemma:DensityIteration}(1) and \eqref{FoFe}, to show
  correlation decay  is equivalent to showing $F_o = F_e$. 
Now, by \eqref{eq:dominance} we have $F_o(\ex) \geq F_e (\ex)$ for
every $\ex \in [0,1]$. 
Hence, if $C_e = C_o=C$, then by
\eqref{relationFo}-\eqref{relationFe}, 
 we have for $\ex \in [0,1]$, 
\[ F_o( \ex) = C\int_{[0,z]} F_e^\Delta (1-t) \mu(dt) \leq
C\int_{[0,z]} F_o^\Delta (1-t) \mu (dt) = F_e(\ex). 
\]
Together with the observation that $F_0 (z) \leq F_z(z)$, this implies
$F_o = F_e$. 
}
\end{remark}

\subsection{Differential equations for $F_o$ and $F_e$}
\label{subs-diffeq}

To show $C_e = C_o,$  we first derive some differential equations for
the functions $F_o$
and $F_e$.    The first result of this section is as follows.

\begin{prop}
\label{prop-ode2}
Suppose the free spin measure $\mu$ is absolutely continuous with 
 density $\rn$ and satisfies $\mu[0,1/2] > 0$.   Let ${\mathcal I}$ be any 
non-empty open set in $[0,1]$ that is symmetric in the sense that $x
\in {\mathcal I}$ implies $1-x \in {\mathcal I}$. 
If $\rn$ is continuously
differentiable and strictly positive on ${\mathcal I}$, then on  ${\mathcal I}$,
the function $F_o$ defined in \eqref{FoFe} is twice continuously
differentiable and satisfies 
\begin{equation}
\label{ode-second}
\ddot F_o (\ex)  =  \frac{\dot{\rn}(\ex)}{\rn(\ex)}\dot F_o(\ex) -
C_oC_e^{1\over \Delta}\Delta  (\rn (\ex))^{1 \over \Delta} \rn(1-\ex) (\dot
F_o(\ex))^{\Delta-1\over \Delta} (F_o(\ex))^{\Delta}, 
\end{equation}
\end{prop}
\begin{proof} 
 Relations 
\eqref{relationFo} and \eqref{relationFe} of Corollary
\ref{cor:Limits} imply that $F_o$
 and $F_e$ are absolutely continuous with density 
$C_e F_e^{\Delta} (1 - \cdot) \rn (\cdot)$ and $C_o F_o^{\Delta} (1 -
\cdot) \rn (\cdot)$, respectively.  Now, if $\rn$ is continuous on
${\mathcal I}$, then clearly, these densities are continuous, and so 
$F_o$ and $F_e$ are continuously differentiable on ${\mathcal I}$.  
If ${\mathcal I}$ is symmetric, then $F_o(1-\cdot)$ and $F_e(1-\cdot)$
are also continuously differentiable and so, if $\rn$ is continuously
differentiable on ${\mathcal I}$, then $F_o$ and $F_e$ are twice
continuously differentiable on ${\mathcal I}$ and for 
 $\ex\in {\mathcal  I}$, 
\begin{eqnarray*}
\ddot F_o (\ex)  =  C_e  \dot \rn (\ex) F_e^\Delta(1-\ex) - C_e\rn (\ex)
\Delta F_e^{\Delta-1} (1-\ex) \dot F_e (1-\ex). 
\end{eqnarray*}
Applying \eqref{relationFo} and \eqref{relationFe} again,  we also have
\begin{align*}
F_e(1-\ex)&= (\dot F_o(\ex))^{1\over \Delta}(C_e m (\ex))^{-1 \over \Delta},\\
\dot F_e (1-\ex)&=C_o\rn (1-\ex)F_o^\Delta(\ex).
\end{align*}
Substituting these identities into the previous expression for  $\ddot F_o$, we obtain the
following second-order ODE  for $F_o$ on ${\mathcal I}$: 
\begin{align*}
\ddot F_o (\ex)  &=  \frac{\dot{\rn}(\ex)}{\rn(\ex)} \dot F_o(\ex) -
C_e \rn(\ex)  \Delta  (C_e \rn (\ex))^{-{\Delta-1\over \Delta}} (\dot F_o(\ex))^{\Delta-1\over \Delta}
C_o \rn (1-\ex) (F_o(\ex))^\Delta\\
&=  \frac{\dot{\rn}(\ex)}{\rn(\ex)}\dot F_o(\ex) -
C_oC_e^{1\over \Delta}\Delta  (\rn (\ex))^{1 \over \Delta} \rn(1-\ex) (\dot
F_o(\ex))^{\Delta-1\over \Delta} (F_o(\ex))^{\Delta}, 
\end{align*}
for $\ex \in {\mathcal I}$. 
\end{proof}

We now fix $\lambda > 0$ and $\ve \in (0,1/2]$, and specialize to the
case when the density $\rn$ of the free spin measure has the form 
\begin{equation}
\label{rn-form}
 \rn (\ex) =  \lambda^{\ex} 
 \ind_{(0,\ve] \cup   [1-\ve, 1)}(\ex), \quad \ex \in [0,1]. 
\end{equation}
Note that the case $\ve = 1/2$ corresponds to the continuous hardcore
model.    Define 
\begin{align}
\theta_o&\mean 
\left(\lambda C_o^{1\over \Delta}C_e\right)^{\Delta\over  \Delta^2-1} \label{thetaoe}
\qquad \mbox{ and } \qquad 
& \theta_e&\mean \left(\lambda C_e^{1\over    \Delta}C_o\right)^{\Delta\over \Delta^2-1}. 
\end{align}

We then have the following result: 
\begin{prop}
\label{prop-hamiltonian}
Suppose the free spin measure has a density $\rn$ of the form
\eqref{rn-form} for some $\ve \in (0,1/2]$ and $\lambda > 0$. Then 
 $F_o$ is twice continuously differentiable on the intervals $(0,\ve)$
and $(1-\ve,1)$ and the function 
\[ \AR _\lambda(z) \doteq \lambda^{-z} (\theta_e F_o(z))^{\Delta+1} +
\lambda^{-\frac{z}{\Delta}}
 (\theta_e \dot{F}_o(z))^{\frac{\Delta+1}{\Delta}}- 
 (\ln \lambda) \lambda^{-\frac{z}{\Delta+1}}
\theta_e^{\frac{\Delta+1}{\Delta}}
F_o(z) (\dot{F}_o(z))^{\frac{1}{\Delta}}, 
\]
is constant on each of the intervals $(0,\ve)$ and $(1-\ve, 1)$. 
Moreover,  $F_o$ satisfies 
 \begin{equation}
 \label{genbc-2}
 \dot{F}_o(0+) = C_e , \qquad  \dot{F}_o(1-) = 0, 
\end{equation}
and 
\begin{equation}
\label{genbc-3} 
\inf\{t > 0: F_o(t) = 1 \} = 1, 
\end{equation}
and 
$\AR_\lambda$ satisfies the boundary conditions
\begin{equation}
\label{AR-bc}
 \AR_\lambda (0+) = (\theta_e C_e)^{\frac{\Delta+1}{\Delta}}, \qquad
\AR_\lambda (1-) = \lambda^{-1} \theta_e^{\Delta+1}. 
\end{equation}
\end{prop}
\begin{proof}
Since the  the density $\rn$ in \eqref{rn-form} is continuously differentiable
on the intervals $(0,\ve)$ and $(1-\ve,1)$, and the 
corresponding free spin measure puts strictly positive mass on
$[0,1]$, 
 it follows from
Proposition \ref{prop-ode2} that $F_o$ is twice continuously
differentiable and satisfies \eqref{ode-second} on each of those
intervals. 
The proof of the first assertion of the proposition proceeds in three steps.  \\
{\em Step 1. } We first  recast the second-order ODE for $F_o$ in \eqref{ode-second} as a system of
non-autonomous first-order ODEs. 
Consider $g(\ex) =
(g_1(\ex),g_2(\ex))$ $\mean (F_o(\ex),\dot
F_o(\ex))$, which lies in $\R_+^2$ since $F_o$ is nonnegative and
nondecreasing.  
Let ${\mathcal I} = (0,\ve) \cup (1-\ve,1)$ if $\ve < 1/2$ and 
let ${\mathcal I} = (0,1)$ if $\ve = 1/2$. 
Since $\rn$ is  continuously differentiable and $\rn(\ex) = (\ln
\lambda) \lambda^\ex$ on ${\mathcal I}$, $F_o$  satisfies 
the second-order ODE in \eqref{ode-second}, which is equivalent to
saying that $g$ satisfies the following system of {\em non-autonomous} first-order
ODEs on ${\mathcal I}$: 
\begin{align}\label{eq:G1}
\dot g(\ex)=G(g(\ex),\ex) \mean (G_1(g_1(\ex), g_2(\ex), \ex),
G_2(g_1(\ex), g_2(\ex), \ex)), 
\end{align}
where for $i = 1, 2,$ $G_i:\R_+^3\rightarrow \R$ are  defined by 
\begin{align}
G_1(y_1,y_2,\ex)&\mean y_2 \label{eq:G2}\\
G_2(y_1,y_2,\ex)&\mean (\ln\lambda)y_2 -
C_oC_e^{1\over \Delta}\Delta  \lambda^{{\ex\over \Delta}} \lambda^{1-\ex}y_2^{\Delta-1\over \Delta}y_1^\Delta, \label{eq:G3}
\end{align}

\noindent 
{\em Step 2. } 
Next, we reparametrize the system of ODEs above to eliminate the
explicit dependence of $G_2$ on $\ex$ in \eqref{eq:G3}. 
Namely, we reformulate the system of ODEs as an \emph{autonomous} system. 
Consider the transformation $\Lambda: (g_1,g_2)\mapsto (h_1,h_2)$ defined by
\begin{align}
h_1(\ex)&=\lambda^{-{\ex\over \Delta+1}}\theta_e g_1(\ex) \label{eq:h1}\\
h_2(\ex)&=\lambda^{-{\ex\over \Delta(\Delta+1)}}\theta_e^{1\over\Delta} (g_2(\ex))^{1\over\Delta}. \label{eq:h2}
\end{align}
We now claim that  $(h_1(\cdot),h_2(\cdot))$
satisfies the following system of ODEs: 
\begin{align}\label{eq:H1}
\dot h(\ex)=H(h(\ex)),  \quad \ex \in (0,1), 
\end{align}
where $H:\R_+^2 \rightarrow \R^2$ is defined by
$H(y_1,y_2)\mean (H_1(y_1,y_2),H_2(y_1,y_2))$, with 
\begin{align}
H_1(y_1,y_2)&\mean -{\ln\lambda\over \Delta+1}y_1+y_2^\Delta, \label{eq:H2}\\
H_2(y_1,y_2)&\mean {\ln\lambda\over \Delta+1}y_2-y_1^\Delta. \label{eq:H3}
\end{align}
The proof is obtained using a fairly straightforward verification. For
$z \in (0,1)$, we have using \eqref{eq:h1}-\eqref{eq:h2}, 
and $\dot g_1(\ex)=g_2(\ex)$ from
\eqref{eq:G1}-(\ref{eq:G2}), 
\begin{align*}
\dot h_1(\ex)&=-{\ln\lambda\over\Delta+1}\lambda^{-{\ex\over \Delta+1}}\theta_e g_1(\ex)+
\lambda^{-{\ex\over \Delta+1}}\theta_e g_2(\ex)\\
&=-{\ln\lambda\over\Delta+1}h_1(\ex)+(h_2(\ex))^\Delta.
\end{align*}
This verifies (\ref{eq:H2}).
Similarly, applying  (\ref{eq:G1}) and (\ref{eq:G3}) together with
\eqref{eq:h1}-\eqref{eq:H1}, and \eqref{eq:H3} we obtain
\begin{align*}
\dot h_2(\ex)&=-{\ln\lambda\over\Delta(\Delta+1)}\lambda^{-{\ex\over
    \Delta(\Delta+1)}}\theta_e^{1\over\Delta}
(g_2(\ex))^{1\over\Delta} \\
&\quad +\lambda^{-{\ex\over \Delta(\Delta+1)}}\theta_e^{1\over\Delta}\Delta^{-1} (g_2(\ex))^{(1-\Delta)\over\Delta}  
(\ln\lambda)g_2(\ex)\\
& \quad  -\lambda^{-{\ex\over \Delta(\Delta+1)}}\theta_e^{1\over\Delta}\Delta^{-1} (g_2(\ex))^{(1-\Delta)\over\Delta} 
C_oC_e^{1\over \Delta}\Delta  \lambda^{{\ex\over
    \Delta}}\lambda^{1-\ex} (g_2(\ex))^{(\Delta-1)\over\Delta}  
(g_1(\ex))^\Delta \\
&=-{\ln\lambda\over\Delta(\Delta+1)}h_2(\ex)
+{\ln\lambda\over\Delta} h_2(\ex)-
\theta_e^{1\over\Delta}\C_oC_e^{1\over \Delta}\lambda
\theta_e^{-\Delta}(h_1(\ex)) ^\Delta \\
&={\ln\lambda\over\Delta+1}h_2(\ex)-(h_1(\ex)) ^\Delta,
\end{align*}
where the last equality uses  definition \eqref{thetaoe} 
of $\theta_e$. 
This verifies (\ref{eq:H3}).

\noindent 
{\em Step 3. } 
Next, we show that the system (\ref{eq:H1})-(\ref{eq:H3}) is a
Hamiltonian system of ODEs, in the sense that if 
$h(\cdot)=(h_1(\cdot),h_2(\cdot))$ is a solution of 
(\ref{eq:H1})-(\ref{eq:H3}) on some interval, then 
  the function 
\[ \Phi\circ h( \ex) = (h_1(\ex))^{\Delta+1}+(h_2(\ex))^{\Delta+1}-(\ln\lambda)
h_1(\ex)h_2(\ex) \] 
is constant on that interval, 
where $\Phi:(y_1, y_2) \mapsto \R$ is defined by 
\[ \Phi (y_1, y_2) \mean y_1^{\Delta+1}  + y_2^{\Delta+1} - (\ln
\lambda) y_1 y_2. 
\]
Indeed, note that on substituting the
expressions for $\dot{h}_1$ and $\dot{h}_2$ obtained above, we have on
this interval, 
\begin{align*}
\dot\Phi \circ h &=(\Delta+1)h_{1}^{\Delta}\dot h_{1}+(\Delta+1)h_{2}^{\Delta}\dot h_{2}-(\ln\lambda)(\dot h_{1}h_{2}+h_{1}\dot h_{2}) \\
&=(\Delta+1)h_{1}^{\Delta}\left(-{\ln\lambda\over \Delta+1}h_1+h_2^\Delta\right)+
(\Delta+1)h_{2}^{\Delta}\left({\ln\lambda\over \Delta+1}h_2-h_1^\Delta\right) \\
&\quad \,  -(\ln\lambda)\left(-{\ln\lambda\over \Delta+1}h_1+h_2^\Delta\right)h_{2}-
(\ln\lambda)\left({\ln\lambda\over \Delta+1}h_2-h_1^\Delta\right)h_{1} \\
&=0.
\end{align*}
The first assertion of the proposition then follows on 
substituting the definition of $h_i$ and $g_i$, $i = 1, 2,$ from Steps
1 and 2 into 
the expression for $\Phi \circ h$ in Step 3.  

Next, note that the boundary conditions in \eqref{genbc-2} follows on substituting the form \eqref{rn-form} of $\mu$
into \eqref{relationFo}.   
When combined with the boundary condition
$F_o(0+)  = F_o(0) = 0$ and $F_o(1-) = F_o(1) = 1$ from Corollary
\ref{cor:Limits}, this implies \eqref{AR-bc}.   
Finally, define $\tau  = \inf \{t > 0: F_o(t) = 1\}$.   Then $F_o(1) =
1$ implies that $\tau \leq
1$.    But one must
have $\dot{F}_o(z) = 0$ for $z > \tau$.  Thus, to prove 
\eqref{genbc-3} it suffices to show that $\dot{F}_o(z) > 1$ for all $z
< 1$. 
Now,  by \eqref{relationFo} 
for $z \in (0, 1)$, 
 $\dot{F}_o(z) = C_e F_e^{\Delta} (1-z) \lambda^z$. 
However, this is strictly
 positive 
because   by symmetry and
 \eqref{genbc-2} it follows that $\dot{F}_e (0) = C_o > 0$, and hence, 
 $F_e (1-z) > 0$ for all $0 < z <  1$. 
This establishes \eqref{genbc-3} and concludes the proof. 
\end{proof}

\subsection{Proof of Uniqueness of Gibbs Measures}
\label{subs-pf}

By Remark \ref{rem-whatsleft}, to prove Theorems 
\ref{theorem:MainResult} and \ref{theorem:MainResult2}, it suffices 
to show that the constants $C_{o}$ and $C_{e}$ in \eqref{eq:C_o} and 
\eqref{eq:C_e}, respectively, are
equal when $\rn$ is given by \eqref{rn-form}, with 
$\ve = 1/2$ and $\ve \in (0,1/2)$, respectively.
In each case, we will use the invariance property 
in Proposition \ref{prop-hamiltonian} to establish this equality.

\begin{proof}[Proof of Theorem~\ref{theorem:MainResult}]
Set $\ve = 1/2$.  Then $\rn$ is continuously differentiable on $(0,1)$
and the function $\AR$ in  Proposition \ref{prop-hamiltonian} is constant 
on the entire interval $(0,1)$. 
 Thus,  setting  $\AR(0+) = \AR(1-)$ in 
 \eqref{AR-bc},  we conclude that  $\theta_e^{\frac{\Delta^2-1}{\Delta}}
= \lambda C_e^{\frac{\Delta+1}{\Delta}}$.  Substituting the value of
  $\theta_e$ from \eqref{thetaoe} into this equation, one concludes
  that  $C_o = C_e$, which   completes the proof. 
\end{proof}

\begin{proof}[Proof of Theorem~\ref{theorem:MainResult2}]
Now, suppose $\ve \in (0,1/2)$.  Then $\rn$ is continuously
differentiable on the intervals $(0,\ve)$ and $(1-\ve, 1)$ and 
so Proposition \ref{prop-hamiltonian} implies 
\begin{equation}
\label{AR-relations}
\AR(0+) = \AR(\ve-), \qquad \mbox{  and } \qquad  \AR((1-\ve)+) = \AR(1-). 
\end{equation}
On the other hand, since $\rn$ is zero on $(\ve, 1-\ve)$, it follows
from \eqref{relationFo}-\eqref{relationFe} that both 
 $F_o$ and $F_e$  are constant on $(\ve, 1- \ve)$.  
In turn, this implies that 
\[  \dot{F}_o (\ve-) = \frac{C_e}{2 \ve}  F_e^{\Delta} (1-\ve)
\lambda^\ve = \frac{C_e}{2 \ve} F_e^{\Delta} (\ve) \lambda^\ve = \dot{F}_o ((1-\ve)+)
\lambda^{2 \ve - 1}. 
\]
Now, if $\lambda = 1$, then these identities and 
the definition of $\AR$ imply 
that $\AR (\ve+) = \AR ((1-\ve)-)$.   Together with
\eqref{AR-relations} and 
\eqref{AR-bc}  this implies 
\[ (\theta_e C_e)^{\frac{\Delta+1}{\Delta}} =\theta_e^{\Delta+1}
\quad \Leftrightarrow \quad (\theta_e)^{\frac{\Delta^2-1}{\Delta}} = \lambda
C_e^{\frac{\Delta+1}{\Delta}}. 
\]
When combined with \eqref{thetaoe}, this shows that $C_e = C_o$.
\end{proof}

\section{Marginal Distributions of the Continuous Hard-core Model}
\label{sec:char}

\subsection{The case $\lambda = 1$: Proof of Theorem \ref{th:DiffEq}}
\label{eq:DE-Uniqueness}

Note that the problem concerns the one-parameter family of ODEs
\begin{align}\label{eq:Fexplicit}
\dot F_{C} (z)=  b(C, F_{C}(z)), 
\end{align}
where  the parameterized family of drifts $b:[0,\infty) \times [0,1]
\mapsto \R_+$  is given by 
\begin{align}
\label{drift}
 b(C,y)  \mean C\left(1-y^{\Delta+1}\right)^{\Delta\over \Delta+1}. 
\end{align}
For any fixed $C > 0$, the function $y \mapsto b(C,y)$ 
 is a Lipschitz continuous function on $(0,1-\delta)$ for any  $\delta
 \in (0,1)$.  Thus there exists a unique solution $F_C$ to the ODE
 \eqref{eq:Fexplicit}
 with boundary condition 
 \begin{equation}
\label{bc1:Fexplicit}
F_C(0) = 0, 
\end{equation}
 on the interval $[0,\tau_C-\delta)$,  where 
\begin{align}
\label{tauC}
  \tau_C \mean \inf \{ t > 0: F_{C} (t) = 1 \}. 
\end{align}
Here,  the infimum over an empty set is taken to be infinity. 
 Since \eqref{ODE} implies
 that $F_C$ is constant after
 $\tau_C$ (if $\tau_C < \infty$), by continuity there is a unique
 continuous solution $F_C$ to
 \eqref{eq:Fexplicit} and \eqref{bc1:Fexplicit} on $[0,\infty)$. 
 
We now show existence of a $C >0$ for which the unique solution $F_C$ to
\eqref{eq:Fexplicit} and \eqref{bc1:Fexplicit} also 
satisfies the boundary condition
\begin{align}
\label{bc2:Fexplicit}
\tau_C = 1. 
\end{align}  
We fix $\lambda = 1$ and $\Delta \geq 1$ and consider the continuous
hardcore model with parameter $\lambda$ and $\Delta$.   
From the proof of Theorem \ref{theorem:MainResult}, it follows that the
constants $C_o, C_e \in (0,\infty)$ defined in \eqref{eq:C_o} and
\eqref{eq:C_e}, respectively, are equal. We denote the common
value by $C_{\Delta,1}$, and  let $\Theta_{\Delta,1}$ denote the corresponding common value
of $\theta_e = \theta_o$ in \eqref{thetaoe}.  Further, let 
$F_{\Delta,1}$ denote the corresponding 
$F_e$, which coincides with $F_o$ by Remark \ref{rem-whatsleft}.  
By  Proposition \ref{prop-hamiltonian}, we have 
\[ (\Theta_{\Delta,1} C_{\Delta,1})^{\frac{\Delta+1}{\Delta}} = R_1(0) =  R_1(z) 
= (\Theta_{\Delta,1} (\Theta_{\Delta,1} F_{\Delta,1}(z))^{\Delta+1} + (\Theta_{\Delta,1} \dot{F}_{\Delta,1}
(z))^{\frac{\Delta+1}{\Delta}}, 
\] 
for every $z \in (0,1)$. 
Noting from  \eqref{thetaoe} that $\Theta_{\Delta,1} =
C_{\Delta,1}^{\frac{1}{\Delta-1}}$ and rearranging terms above, this
implies that $F_{\Delta,1}$ satisfies the ODE \eqref{ODE} 
when $C = C_{\Delta,1}$.  
Furthermore, $F_{\Delta, 1} (0) = 0$ by Corollary \ref{cor:Limits} and
hence, $F_{\Delta,1}$ 
is the unique solution  $F_{C_{\Delta,1}}$ to  \eqref{eq:Fexplicit} and 
\eqref{bc1:Fexplicit}.   Furthermore, it follows from 
\eqref{genbc-3} that $ \tau_{C_{\Delta,1}} = 1, $ and thus we have
shown that 
\eqref{eq:Fexplicit}, \eqref{bc1:Fexplicit} and \eqref{bc2:Fexplicit} are satisfied when $C =C_{\Delta,1}$.

To prove Theorem \ref{th:DiffEq},  it only remains to prove that 
there is a  unique constant $C$ (equal to $C_{\Delta,1}$) for 
which the  unique solution $F_C$ to  \eqref{eq:Fexplicit} and 
\eqref{bc1:Fexplicit} also satisfies  \eqref{bc2:Fexplicit}. 
Our next result shows that this is the case. 

\begin{prop}\label{prop:Funique}
Given $C > 0$, let $F_C$ be the unique solution to
\eqref{eq:Fexplicit} and \eqref{bc1:Fexplicit}  on $[0,\infty)$, 
and define $\tau_C$ as in \eqref{tauC}.  
 The function  $[0,\infty) \ni C \mapsto \tau_{C}$ is strictly
 decreasing and continuous with range $\R_{+}$. 
  In particular, there exists a unique
$C^{*} > 0$ such that $\tau_{C^{*}}=1$.
\end{prop}
\begin{proof}  

The proof entails 
 two main steps. \\
{\em Step 1: }  We show $\lim_{C \downarrow 0} \tau_C =
\infty$ and $\lim_{C \uparrow \infty} \tau_C = 0$. \\
First, observe that $\tau_{C}\ge 1/C$ since $F(0)=0$ and $\dot F(z)
\le C$ for all $z$. 
Thus $\tau_{C}\rightarrow\infty$ as $C\rightarrow 0$.  
Next, set $\sigma_{C}(0)  \mean 0$ and for $\delta > 0$, define 
\begin{equation}
\label{sigmaC}
\sigma_{C}(\delta) \mean \inf \left\{ z > 0: F_C (z) = 1 - \delta  \right\}. 
\end{equation}
Observe that the set of such $z$ is non-empty since for every $z$ such that $F_C(z)<1-\delta$ 
we have the uniform lower bound $\dot{F}_C(z)>C(1-(1-\delta)^{\Delta+1})^{\Delta\over\Delta+1}>0$. 
Now,  for $z \in [\sigma_C(1/(n-1)),
\sigma_C(1/n)]$, $\frac{n-2}{n-1} \leq F_C(z) \leq \frac{n-1}{n}$,
and hence, 
\[   \dot{F}_C(z) =  C \left( 1 - F_C^{\Delta+1}(z)\right)^{\Delta
  /(\Delta+1)}   = 
\frac{C(\Delta+1)^{\frac{\Delta}{\Delta+1}}}{n^{\frac{\Delta}{\Delta+1}}} +
o \left( \frac{1}{n^{\frac{\Delta}{\Delta+1}}} \right), 
\]
where $o(\ve)$ represents a quantity that vanishes as $\ve \rightarrow
0$. 
Using  the identity 
\[ \frac{1}{n-1} - \frac{1}{n} =  F_C(\sigma_{C} (1/(n-1)) - F_C
(\sigma_C(1/n)) = \int_{\sigma_C(1/(n-1))}^{\sigma_C(1/n)} \dot{F}_C (z)
dz, 
\]
we obtain the estimate 
\begin{equation}
\label{sigma-diff}
 \sigma_{C} \left( \frac{1}{n}\right)  - \sigma_C \left( \frac{1}{n-1}\right)  = \left( C
  (\Delta+1)^{\frac{\Delta}{\Delta+1}}
  n^{\frac{\Delta+2}{\Delta+1}}\right)^{-1} + o\left(
\frac{1}{n^{\frac{\Delta+2}{\Delta+1}}} \right). 
\end{equation}
In turn, since $\tau_C =  \sum_{n=1}^\infty \left[\sigma_{C}\left(\frac{1}{n}\right) -
  \sigma_{C}\left(\frac{1}{n-1}\right)\right]$, this  implies that 
\begin{align*}
C\tau_C =  (\Delta+1)^{-\frac{\Delta}{\Delta+1}} \sum_{n=1}^\infty
n^{-\frac{\Delta+2}{\Delta+1}} + o(1) < \infty, 
\end{align*}
which shows that $\tau_C \rightarrow 0$ as $C \rightarrow \infty$.  
This concludes the proof of Step 1.   

Before proceeding to Step 2, observe that in a similar fashion, for  $\delta \in
[1/n, 1/(n-1)]$, we have 
\begin{align*}  
\sum_{m \geq n}  \left[\sigma_C \left(  \frac{1}{m}\right) - \sigma_C
  \left(  \frac{1}{m-1}\right) \right]&\leq \tau_C - \sigma_C (\delta)
\\
& \leq \sum_{m \geq n-1}  \left[\sigma_C \left(
  \frac{1}{m}\right) - \sigma_C \left(  \frac{1}{m-1}\right) \right]. 
\end{align*}
Combining this with  the estimate \eqref{sigma-diff} and the fact that
for both $k = n$ and 
$k=n+1$, the sum $\sum_{m\ge k} {m^{-{\Delta+2\over \Delta+1}}}$ 
is of the order  $O(\frac{\Delta+1}{\Delta+2}n^{-\frac{1}{\Delta+1}})$, 
we conclude that 
\begin{equation}
\label{rate-delta}
 \tau_C - \sigma_C(\delta) = \frac{(\Delta+1)^{\frac{1}{\Delta+1}}}{C(\Delta+2)} \delta^{\frac{1}{\Delta+1}} +
o(\delta^{\frac{1}{\Delta+1}}). 
\end{equation}

\noindent
{\em Step 2. }  We show  that $C \mapsto \tau_C$ is strictly
decreasing and continuous. \\
 First, note that  for $b$ given in \eqref{drift}, we have 
\begin{align}
\label{pders}
  \frac{\partial b}{\partial C}(C,y) =
(1-y^{\Delta+1})^{\frac{\Delta}{\Delta+1}}, \qquad  \frac{\partial
  b}{\partial y} = -C \Delta (1 - y^{\Delta+1})^{-\frac{1}{\Delta+1}}
y^\Delta. 
\end{align}
Thus,  $b$ is continuously differentiable 
with bounded partial derivatives on $[0,\infty) \times [0,1-\frac{1}{n}]$ for
every $n \in \mathbb{N}$. 
Then, by standard sensitivity analysis for parameterized ODEs
we know that for every $n \in \mathbb{N}$,  on $[0,\sigma_C(\frac{1}{n})]$, $R_C (z) \mean \partial F_C(x)/\partial C$
exists and satisfies 
\begin{align*}
\frac{\partial R_C}{\partial z} (z) = \frac{\partial^2 F_C}{\partial
  C\partial z} (z) &= \frac{\partial}{\partial
  C} b (C, F_C(z)) \\
&= \frac{\partial b}{\partial C} (C, F_C(z)) + \frac{\partial
  b}{\partial y} (C, F_C(z)) \frac{\partial F_C}{\partial C} (z), 
\end{align*}
which yields   the following first-order inhomogeneous linear ODE for
$R_C$: 
\begin{align}
\label{sens-ode} 
 \frac{\partial R_C}{\partial z} (z) &= \frac{\partial b}{\partial C} (C, F_C(z)) + \frac{\partial
  b}{\partial y} (C, F_C(z)) R_C(z). 
\end{align}
Moreover, since $F_C(0) = 0$ for all $C$, $R_C$ satisfies  the boundary condition 
\begin{equation}
\label{sens-bc}
   R_C(0) = 0.
\end{equation}
Solving the linear ODE \eqref{sens-ode}-\eqref{sens-bc}, we
obtain 
\[  R_C (z) = \int_0^z e^{\int_x^z \frac{\partial b}{\partial y} (C,
  F_C(t)) dt} \frac{\partial b}{\partial c} ( C, F_C (x)) \, dx. 
\] 
Substituting the partial derivatives of $b$ from \eqref{pders}, 
we have for  $z  \in (0, \tau_C)$, 
\begin{align}
\label{RC}
 R_C (z)  &= \int_0^z e^{-C\Delta\int_x^z
  (1-(F_C(t))^{\Delta+1})^{-\frac{1}{\Delta+1}} (F_C(t))^\Delta dt}
(1 - (F_C(x))^{\Delta+1})^{\frac{\Delta}{\Delta+1}} dx  > 0. 
\end{align}
Now, fix $n \in\mathbb{N}$ and  recall from \eqref{sigmaC} that $F_C(\sigma_C(1/n)) = 1- 1/n$.  
Since $(C,z)\mapsto F_C(z)$ is continuously differentiable, and by \eqref{eq:Fexplicit}-\eqref{drift} for 
any fixed $C_0 > 0$,  $\frac{\partial F_{C_0}}{\partial z}
(\sigma_{C_0} (1/n)) 
> 0$, it follows from the implicit function theorem 
%
that $C \mapsto \sigma_C(1/n)$ is continuously differentiable, 
and \eqref{RC} then implies that 
\begin{equation}
\label{pder-sigmaC}
\frac{d (\sigma_C(1/n))}{dC} < 0.
\end{equation} 
Now, for any $C < \infty$, fix  $C_- < C < C_+$.   Then for all sufficiently large $n$, it
follows  from \eqref{rate-delta} that 
\[  \tau_{C_+} - \sigma_{C_+} (1/n) \leq \tau_{C} - \sigma_{C}
(1/n) \leq \tau_{C_-} - \sigma_{C_-} (1/n). 
\]
Since \eqref{pder-sigmaC}  implies that $\sigma_{C+} (1/n) < \sigma_C
(1/n) < \sigma_{C-} (1/n)$, this shows that
$\tau_{C-} > \tau_{C} > \tau_{C+}$, 
namely $C \mapsto \tau_C$ is strictly
decreasing on 
$(0,\infty)$. 
Finally, to show that $\tau$ is continuous, fix $C > 0$, and  given 
$\ve  > 0$,  note that \eqref{rate-delta} shows that there exists a
sufficiently large $n$, such that for all $\eta < C$ and 
$\tilde{C} \in [C-\eta, C+\eta]$, 
\[  \left|\tau_{\tilde{C}} - \sigma_{\tilde{C}} (1/n) - (\tau_{C} -\sigma_{C}(1/n))\right| \leq \frac{\varepsilon}{2}. 
\]
Since, as shown above,  $C \mapsto \sigma_{C} (1/n)$ is
continuous (in fact, continuously differentiable),  there exists
$\delta < 1$ such that whenever $|\tilde{C} - C| < \delta$, 
$|\sigma_{\tilde{C}}(1/n) - \sigma_{C} (1/n)| <
\frac{\varepsilon}{2}$, and hence,  $|\tau_{\tilde{C}} - \tau_{C}|
<\varepsilon.$ This shows that $C \mapsto \tau_C$ is continuous, and
concludes the proof of Step 2.

Finally, we note that by Step 1 and the continuity of $\tau_C$
established in Step 2,  $\{\tau_{C}, C \in (0,\infty)\} =
(0,\infty)$. 
Since $C \mapsto \tau_{C}$ is a strictly decreasing continuous
function by Step 2, this implies the  existence of a unique $C^{*}$ with $\tau_{C^*}
= 1$. This completes the proof of Proposition~\ref{conjecture:DiffEq}.
\end{proof}

\subsection{A Conjecture for General $\lambda > 0$} 
\label{subs-conjecture}

Fix $\lambda > 0, \Delta \geq 1$, and let $F = F_{\Delta,\lambda} = F_o$ and
$C = C_o = C_e$ be 
the limiting function and constant, respectively,   
from Theorem \ref{theorem:MainResult}. 
Then the free spin measure with density $\rn (z) = \lambda^z$ satisfies the conditions 
of Proposition \ref{prop-ode2} with ${\mathcal I} = (0,1)$ and so it follows from 
 \eqref{ode-second} that $F$ satisfies the second-order ODE 
\be
\label{eqn1.0}
\ddot{F}(z) =    (\ln\lambda) \dot F (z) - C^{{1\over \Delta}+1} \Delta \lambda^{(1-z)}
\lambda^{z/\Delta}  \left(\dot F (z)\right)^{1 - \frac{1}{\Delta}} F^\Delta (z),
\ee  
for $z \in (0,1)$.   Moreover, Corollary \ref{cor:Limits},
\eqref{genbc-2} and \eqref{genbc-3} show that $F$ also satisfies the boundary
conditions  
\begin{equation}\label{bc.1}
\left\{ 
 \begin{array}{l}
  \, F(0) = 0, \\
  \, \inf \{ t  > 0: F(t) = 1 \}  = 1   \\
  \, \dot F(0+) = C,  \\
  \, \dot F(1) = 0. 
 \end{array}
\right. 
\end{equation}

We conjecture the following generalization of Theorem
\ref{th:DiffEq} holds, but defer  investigation of its validity to future work. 

\begin{conjecture}\label{conjecture:DiffEq}  
There exists a unique $C_{\Delta, \lambda} > 0$ for which the
ODE \eqref{eqn1.0}-\eqref{bc.1} admits a solution, and 
$F_{\Delta, \lambda}$ is a twice continuously differentiable function
that is the unique solution to  \eqref{eqn1.0}-\eqref{bc.1} with 
$C = C_{\Delta, \lambda}$. 
\end{conjecture} 

\section{Graphs with Large Girth}\label{section:RegularGraphs}

We now switch our focus to the problem of computing the volume of the
LP polytope ${\mathcal P} (\G)$ 
of a  $\Delta$-regular graph $\G$ with large
girth, and specifically prove  Theorem~\ref{theorem:MainResultRegularGraph} in Section
\ref{subs-pf-Regulargraph}.   
The proof approach we use follows closely the technique 
used in~\cite{BandyopadhyayGamarnikCounting} for the problem of
counting the asymptotic number of independent sets in regular graphs
with large girth. 
First, in Section \ref{subsection:rewiring} we discuss a certain \emph{rewiring} technique that allows
one to construct  $(N-2)$-node regular graph with large girth from
an $N$-node regular graph with large girth by deleting and adding
only a constant number of 
(specific) nodes and edges. 

\subsection{Rewiring}\label{subsection:rewiring}

Here, we summarize relevant results from
~{\cite[Section 4.3]{BandyopadhyayGamarnikCounting}}. 
Given an $N$-node $\Delta$-regular graph $\G$, fix any two nodes
$u_1, u_2$  such that the graph theoretic distance between $u_1$ and $u_2$ is at least four.
The latter ensures that there are no edges between the 
non-overlapping neighbor sets of $u_1$ and $u_2$, which we denote by $u_{1,1},\ldots,u_{1,\Delta}$
and $u_{2,1},\ldots,u_{2,\Delta}$, respectively. 
Consider a modified graph $\Hg$  obtained  from $\G$ by deleting the
nodes $u_1$ and $u_2$, 
and adding an edge between $u_{1,i}$ and $u_{2,i}$ for every 
$i=1,\ldots,\Delta$; see Figure~\ref{figure:rewire}.  
The resulting graph $\Hg$ is a $\Delta$-regular graph with $N-2$
nodes.   
We call this operation  a ``rewiring'' or ``rewire'' operation. In our application, the rewiring step will be applied only to 
pairs of nodes with distance at least four.
Rewiring was used in \cite{MezardParisiCavity} and \cite{MezardIndSets2004}
in the context of random regular graphs, and it was performed on two nodes selected
randomly from the graph. Here, as in~\cite{BandyopadhyayGamarnikCounting}, we will instead rewire on nodes $u_1$ and $u_2$
that are farthest from each other.  As shown in the next result, this will enable us to
preserve the large girth property of the graph for many rewiring steps.

\begin{figure}\label{figure:rewire}
\begin{center}
\scalebox{.20}{
\includegraphics{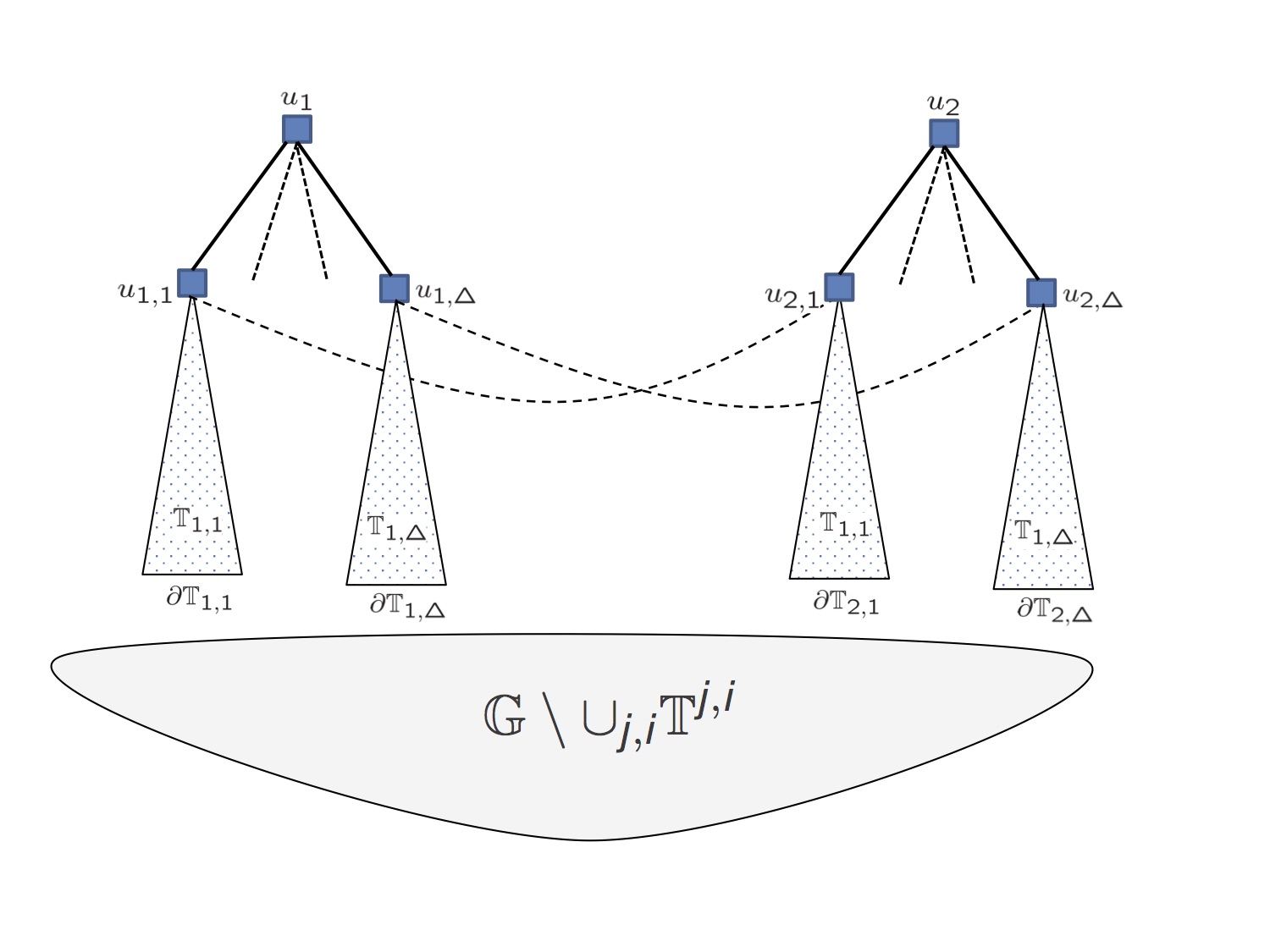}
}
\caption{Rewiring on nodes $u_1$ and $u_2$}
\end{center}
\end{figure}

Recall that $g(\G)$ denotes the girth of the graph $\G$.
We now state  Lemma 2 of
  ~\cite{BandyopadhyayGamarnikCounting}. 
 For completeness we include the proof of this lemma in Appendix \ref{sec-ap1}.

\begin{lemma}\label{lemma:Rewire}
Given an arbitrary $N$-node $\Delta$-regular graph $\G$, consider any
integer $4\leq g\leq g(\G)$. 
 If $2(2g+1) \Delta^{2g} < N$, then 
the rewiring operation can be
performed for at least $(N/2)-(2g+1)\Delta^{2g}$ steps on pairs of
nodes that  are  a distance at least $2g+1$ apart. 
After every rewiring step, the resulting graph is $\Delta$-regular with girth at least $g$.
\end{lemma}

\begin{remark}
\label{rem-rewiring}
{\em 
If the same fixed $g \in \{4, \ldots, g(\G)\}$ is used at each step, 
since every rewiring step reduces the graph by $2$ nodes, we see that after $(N/2)-(2g+1)\Delta^{2g}=N/2-O(1)$ rewiring steps,
the resulting graph is of constant $O(1)$ size, which will have a negligible
contribution to the asymptotic formula for the volume of
$\mathcal{P}(\G)$.}
\end{remark}

\subsection{Proof of Theorem~\ref{theorem:MainResultRegularGraph}}
\label{subs-pf-Regulargraph}

Fix $\lambda  > 0$, $\Delta \geq 1$ and let $F_{\Delta, \lambda}$ be
the distribution function from Theorem \ref{theorem:MainResult}. 
We fix an arbitrary sequence $\G_n, n \in \mathbb{N},$ of $\Delta$-regular graphs with diverging girth: $\lim_{n\rightarrow\infty}g(\G_n)=\infty$.
In what follows,  we adopt the short-hand notation
$x\lessgtr(1\pm \epsilon)y$ to mean  $(1-\epsilon)y\le x\le
(1+\epsilon)y$. 
The  main technical result underlying our proof of Theorem~\ref{theorem:MainResultRegularGraph} is as follows.
\begin{theorem}\label{theorem:rewire}
For every $\Delta \geq 2$, $\epsilon>0$ and $\lambda > 0$, there exists a large enough
integer $g=g(\epsilon,\Delta,\lambda)$  such that if the rewiring is performed
on any $\Delta$-regular graph $\G$ with girth $g(\G)\ge g$ on two
nodes that are at least $2g+1$ distance apart, then for the resulting
graph $\Hg,$ we have 
\begin{align}
{Z_{\G, \lambda} \over Z_{\Hg, \lambda}}&\lessgtr (1\pm \epsilon)
\left(\int_0^1 \lambda^t F^\Delta_{\Delta-1,\lambda}(1-t)dt\right)^{-2}\notag\\
&\quad \times \left(\int_0^1\dot F_{\Delta-1,\lambda}(t)F_{\Delta-1,\lambda}(1-t)dt\right)^{-\Delta}. \label{eq:RHSLimit}
\end{align}
\end{theorem}

We first show how this result implies
Theorem~\ref{theorem:MainResultRegularGraph}.

\begin{proof}[Proof of Theorem~\ref{theorem:MainResultRegularGraph}]
We fix $\epsilon > 0$ and $g=g(\epsilon,\Delta,\lambda) \geq 4$,     
as described in Theorem~\ref{theorem:rewire}. Since 
$g(\G_n)\rightarrow\infty$, 
we have $g(\G_n)\ge g$ for all sufficiently large $n$. 
For $t = 1, \ldots,  \TT_n \mean|V(\G_n)|/2-(2g+1)\Delta^{2g}$,    
let $\G_{n,t}$  be the graph obtained from $\G_{n,0}\mean \G_n$ after $t$
rewiring steps.  Then, for any $\lambda > 0$, trivially we have 
\begin{align*}
Z_{\G_n,\lambda}=\left(\prod_{1\le t\le \TT_n}{Z_{\G_{n,t-1},\lambda}\over Z_{\G_{n,t},\lambda}}\right)Z_{\G_{n,{\TT_n},\lambda}}.
\end{align*}
For conciseness, we introduce the notation  
\begin{align*}
\Gamma(\Delta,\lambda)\mean \left(\int_0^1 \lambda^t F^\Delta_{\Delta-1,\lambda}(1-t)dt\right)^{-2}
\left(\int_0^1 \dot
  F_{\Delta-1,\lambda}(t)F_{\Delta-1,\lambda}(1-t)dt\right)^{-\Delta}, 
\end{align*}
and note that  by Lemma~\ref{lemma:Rewire} and
Theorem~\ref{theorem:rewire}, 
for $1\le t\le \TT_n,$ 
\begin{align*}
{Z_{\G_{n,t-1},\lambda}\over Z_{\G_{n,t},\lambda}}&\lessgtr (1\pm \epsilon)\Gamma(\Delta,\lambda).
\end{align*}
Therefore, we obtain
\begin{align*}
Z_{\G_n,\lambda}&\lessgtr (1\pm \epsilon)^{\TT_n}\Gamma^{\TT_n}(\Delta,\lambda)Z_{\G_{n,\TT_n},\lambda}.
\end{align*}
Now, recall from Remark \ref{rem-rewiring} that the number of nodes, 
and hence  edges, of $\G_{n,\TT_n}$ is bounded 
by a constant that  does not depend on $n$.  In turn, this implies
that $Z_{\G_{n,\TT_n},\lambda}$ is also bounded by a constant that does not
depend on $n$. 
Therefore, taking the natural  logarithm  of both sides of the last display, dividing by
$|V(\G_n)|$,  recalling that $\TT_n=|V(\G_n)|/2-O(1)$, and taking limits, first as $n\rightarrow\infty$ and then
as $\epsilon\rightarrow 0$, we obtain \eqref{eq:VolumeLimit}. 
\end{proof}

\vspace{0.1in}

The remainder of this section is devoted to proving Theorem~\ref{theorem:rewire}.
Fix an integer $g$, and consider an arbitrary $\Delta$-regular graph $\G$ with girth $g(\G)\ge 2g+1$ and
fix any two nodes $u_1$ and $u_2$ in $\G$ that are at least a distance
$2g+1$ apart.  Fix $\lambda > 0$, and 
let $\P = \P_{\mu_\lambda, \G}$ be the continuous hardcore
measure on $\Pol (\G)$, and for any induced subgraph $\tilde{\G}$, let 
$\P_{\tilde{\G}}$ represent the continuous hardcore measure on
$\tilde{\G}$, and let $\E$ and $\E_{\tilde{\G}}$ represent the corresponding
expectations. Also,  as usual, let $\mbX$ be the random vector
representing spin values at nodes. 
  Moreover, let $\Hg$ be the graph
obtained on rewiring on $u_1$ and $u_2$  and 
omitting the dependence on $\lambda$ for conciseness, 
let $Z_{\G}$, $,Z_{\G\setminus \{u_1, u_2\}}$ and $Z_{\Hg}$, respectively, be the partition functions
associated with the continuous hardcore model (with parameter
$\lambda$) on $\G, \G\setminus\{u_1, u_2\}$ and $\Hg$, respectively.  
Next,  for $j=1, 2$, we  denote by $u_{j,1},\ldots,u_{j,\Delta}$
the neighbors of $u_j$ in $\G$, and  let $\T^{j}$
be the subtree of depth $g$ rooted at $u_{j}$, and for $j=1, 2$, let 
$\T^{j,i}$ denote the subtree of $\G \setminus \{u_1, u_2\}$ rooted at
$u_{1,i}$.  Note that  (since $u_j$ has been removed), $u_{j,i}$ has 
$\Delta-1$ children, as do each of the other internal nodes of
$\T^{j,i}$. Thus, 
 for every $i,j$, $\T^{j,i}$ is isomorphic to
$\T_{g-1,\Delta-1}$, denoted $\T^{j,i} \sim \T_{g-1,\Delta-1}$.  Finally,
recall the definition of $F_{n,\Delta} = F_{n,\Delta,\lambda}$ given
in Section \ref{subs-densityevol}.  
The proof of Theorem \ref{theorem:rewire} relies on 
 two preliminary estimates, stated in Lemmas \ref{lem-ratio1} and
 \ref{lem-ratio2} below. 
We first show that these estimates imply  Theorem
\ref{theorem:rewire}, and only then  prove the estimates.

\begin{lemma}
\label{lem-ratio1}
\begin{align}
\label{eq-ratio1}
  \frac{Z_{\G\setminus \{u_1, u_2\}}}{Z_{\G}} &=  \E \left[ 
\prod_{j=1}^2 \left( \int_{t \in [0,1]} \lambda^t
  \prod_{i=1}^\Delta F_{g-1,  \Delta-1}
  (1-t|\mbX_{\partial \T^{j,i}}) dt \right)^{-1} \right]. 
\end{align}
\end{lemma}

\begin{lemma}
\label{lem-ratio2}
\begin{align}
\notag 
  &\frac{Z_{\Hg}}{Z_{\G\setminus \{u_1, u_2\}}} \\
& = 
\E \left[ \prod_{i=1}^\Delta \int_{t \in [0,1]} dF_{g-1,
    \Delta-1}    (t|\mbX_{\partial \T^{1,i}}) F_{g-1, \Delta-1}
  (1-t|\mbX_{\partial \T^{2,i}})\right]. 
\label{eq-ratio2}
\end{align}
\end{lemma}

\begin{proof}[Proof of Theorem \ref{theorem:rewire}]
We  first use Theorem~\ref{theorem:MainResult} to approximate the right-hand sides of \eqref{eq-ratio1} and
\eqref{eq-ratio2} for large $g$. 
Let $F_{\Delta-1} = F_{\Delta-1, \lambda}$ be the limit function in
Theorem~\ref{theorem:MainResult}. 
Then, given $\epsilon > 0$, by Theorem~\ref{theorem:MainResult} and the bounded convergence
theorem, 
 for sufficiently large $g$ and every boundary condition
 $\mbx_{\partial\T^{j,i}} \in \Pol (\partial \T^{j,i})$, $1\le i\le
 \Delta,$ $j= 1, 2$, 
\begin{align}
\label{est1}
\ds 1-\dfrac{\epsilon}{8}\le 
{\int_{t\in [0,1]}\lambda^t F^\Delta_{\Delta-1}(1-t)dt
\over
\int_{t \in [0,1]} \lambda^t
  \prod_{i=1}^\Delta F_{g-1,  \Delta-1}
  (1-t|\mbx_{\partial \T^{j,i}}) dt }
\le 1+ \frac{\epsilon}{8}, 
\end{align} 
Next, note that Theorem~\ref{theorem:MainResult} implies that the
probability measure $dF_{g-1,\Delta-1} (\cdot| \mbx_{\partial \T^{1,i}})$ converges to
the probability measure  $dF_{\Delta-1}(\cdot)$ 
in the Kolmogorov distance (and therefore the L\'{e}vy
distance, which induces weak convergence on $\R$), uniformly with
respect to all feasible boundary conditions (see \cite[Chapter 2]{HubRon09}
for definitions of the Kolmogorov and L\'{e}vy distances and the
relation between them). 
Since, again by Theorem \ref{theorem:MainResult}, $F_{g-1,\Delta-1} (\cdot|\mbx_{\partial \T^{2,i}})$, $g \in \N,$ is a
sequence of bounded continuous functions that converges uniformly to 
$F_{\Delta-1} (\cdot)$, also uniformly with respect to the boundary
condition $\mbx_{\partial \T^{2,i}} \in \Pol (\partial \T^{2,i})$, this shows that there exists $g$ large enough
such that  
\begin{equation} 
\label{est2}
1 - \frac{\epsilon}{8} \leq \displaystyle \dfrac{\left(\int_{t \in [0,1]}
    dF_{\Delta-1}    (t) F_{\Delta-1}
    (1-t)\right)^\Delta}{\prod_{i=1}^\Delta \int_{t \in [0,1]}
  dF_{g-1,  \Delta-1}    (t|\mbx_{\partial \T^{1,i}}) F_{g-1,
    \Delta-1}  (1-t|\mbx_{\partial \T^{2,i}})}  
\leq 1 + \frac{\epsilon}{8}, 
\end{equation}
for all boundary conditions $\mbx_{\cup_{j=1}^2 \partial T^j}$. 
Now,  fix $\ve > 0$ sufficiently small such that $(1-\ve) \leq (1-\ve/8)^3$ and
$(1+\ve/8)^3 \leq (1+\ve)$,  and $g= g(\ve, \Delta, \lambda)$ sufficiently large
such that  \eqref{est1} and \eqref{est2} hold.  
Given the uniformity in these estimates with respect to boundary
conditions, \eqref{est1} and \eqref{est2} also hold when  $\mbx$ is
replaced by $\mbX$.  Now,  taking the
product of the middle term in \eqref{est1} for $j=1, 2$ and then
taking expectations, and taking expectations of the denominator in
\eqref{est2}, combining this with 
\eqref{eq-ratio1}-\eqref{eq-ratio2},  and using the fact that $F_{\Delta-1} 
= F_{\Delta-1,\lambda}$ is absolutely continuous to write
$dF_{\Delta-1}(t) = \dot{F}_{\Delta-1} (t) dt$,  we obtain 
\eqref{eq:RHSLimit}. 
This completes the proof of Theorem \ref{theorem:rewire}. 
\end{proof}

We now turn to the proofs of the lemmas, starting with Lemma
\ref{lem-ratio1}. 

\begin{proof}[Proof of Lemma \ref{lem-ratio1}]  For notational
  conciseness, set $\tilde{\G} \mean \G \setminus \{u_1, u_2\}$.  Then,  since $u_1, u_2$ are not
neighbors, for $z_1, z_2 \geq 0$,
\begin{align*}
& \P( X_{u_1} \leq z_1, X_{u_2} \leq z_2) \\
& =Z_{\G}^{-1}\int_{x_{u_1} \in
  [0,z_1]} \lambda^{x_{u_1}}\int_{x_{u_2} \in [0,z_2]} \lambda^{x_{u_2}}
\int_{\mbx_{\tilde{\G}} \in \Pol (\tilde{\G}): x_{u_{j,i}} + x_{u_j}
  \leq  1 \forall i, j}
 \prod_{u\in V(\tilde{\G})} \lambda^{x_u} d\mbx,
\end{align*}
where the range of $i, j$ above is $i = 1, \ldots, \Delta$ and $j=1,
2$. 
Then 
\begin{align}
 \frac{\partial^2}{\partial z_1 \partial z_2} \P( X_{u_1} \leq z_1,
X_{u_2} \leq z_2)\Big|_{z_1 \downarrow 0, z_2 \downarrow 0} 
& = Z_{\G}^{-1} \int_{\mbx_{\tilde{\G}} \in \Pol (\tilde{\G})}
 \prod_{u\in V(\tilde{\G})} \lambda^{x_u} d\mbx_{\tilde{\G}} \notag \\
&=  Z_{\G}^{-1} Z_{\tilde{\G}}.  \label{ratio-pder}
\end{align}
Since the trees $\T^j, j = 1,2,$ are non-intersecting and each  $u_j$
lies in  $\T^j$, using the spatial
Markov property in the second equality below, we have
\begin{align*}
\P ( X_{u_1} \leq z_1, X_{u_2} \leq z_2) &= \E\left[ \P \left( X_{u_1} \leq z_1,
X_{u_2} \leq z_2|\mbX_{\cup_{j=1}^2 \partial \T^j} \right)\right] \\
& = \E \left[ \prod_{j=1}^2 \P \left( X_{u_j} \leq z_j|\mbX_{\partial
      \T^j}\right)\right]. 
\end{align*}
Now, fix $j\in \{1,2\}$. Then $\T^j$ is not isomorphic to $\T_{g,
  \Delta-1}$,  but each of the disjoint trees $\T^{j,i}, i = 1,
\ldots, \Delta,$ rooted at the corresponding neighbor
$u_{j,i}$ of $u_j$ are isomorphic to 
$\T_{g-1,\Delta-1}$.   Thus, another application of the spatial Markov
property shows that 
\begin{align*}
  \P ( X_{u_j} \leq z_j|\mbX_{\partial   \T^j}) & = \int_{x_{u_j}    \in [0,z_j]} \lambda^{x_{u_j}} 
\prod_{i=1}^\Delta \P\left( X_{u_{j,i}} \leq 1 - x_{u_j}| \partial
\mbX_{\partial \T^{j,i} }\right) dx_{u_j} \\
& \quad \times \left(\int_{x_{u_j}      \in [0,1]} \lambda^{x_{u_j}} 
\prod_{i=1}^\Delta \P\left( X_{u_{j,i}} \leq 1 - x_{u_j}| \partial
\mbX_{\partial \T^{j,i} }\right) dx_{u_j} \right)^{-1}. 
\end{align*}
Taking the derivative with respect to $z_j$, we get 
\begin{align*}
\frac{d}{d z_j}  \P ( X_{u_j} \leq
  z_j|\mbX_{\partial   \T^j}) \Big|_{z_j \downarrow 0}  &=\left( \int_{t    \in [0,1]} \lambda^{t} 
\prod_{i=1}^\Delta \P\left( X_{u_{j,i}} \leq 1 -t| \partial
\mbX_{\partial \T^{j,i} }\right) dt \right)^{-1} \\
&=\left( \int_{t    \in [0,1]} \lambda^{t} 
\prod_{i=1}^\Delta F_{g-1, \Delta-1} (1-t|\mbX_{\partial \T^{j,i} }dt \right)^{-1}, 
\end{align*}
where the last equality uses the fact that  $\T^{j,i} \sim
\T_{g-1,\Delta-1}$ and $u_{j,i}$ is its root. 
The last four displays, together with the dominated convergence theorem
(to justify interchange of $\E$ and differentiation $d/dz_j$)
and \eqref{ratio-pder}, yield \eqref{eq-ratio1}. 
\end{proof}

\begin{proof}[Proof of Lemma \ref{lem-ratio2}]
Recall that $\Hg$ is the
graph obtained from $\G\sm \{u_1,u_2\}$ by adding edges between
$u_{1,i}$ and $u_{2,i}$ for every $i=1,\ldots,\Delta$. 
Thus, 
\[  \Pol (\Hg) = \{ \mbx \in \Pol (\G\sm \{u_1, u_2\}): x_{u_{1,i}} +
x_{u_{2,i}} \leq 1,  1  \leq  i \leq \Delta \},
\]
and hence, 
\begin{align}
\label{eq:Ztilde}
\frac{Z_{\Hg}}{Z_{\G\sm \{u_1,u_2\}}} &=  \pr_{\G\sm  \{u_1,u_2\}}(X_{u_{1,i}}+X_{u_{2,i}}\leq 1, ~ 1\le i\le \Delta). 
\end{align} 
The right-hand side of \eqref{eq:Ztilde} above can be rewritten as 
\begin{align*}
& \P_{\G\sm \{u_1,u_2\}}(X_{u_{1,i}}+X_{u_{2,i}}\leq 1, ~ 1\le i\le
\Delta) \\
& = \E_{\G\sm \{u_1,u_2\}} \left[ \P_{\G\sm \{u_1,u_2\}}\left(X_{u_{1,i}}+X_{u_{2,i}}\leq 1, ~ 1\le i\le
\Delta | \mbX_{\cup_{i=1}^\Delta \cup_{j=1}^2  \partial \T^{j,i}} \right)\right]. 
\end{align*}
Since $u_1$ and $u_2$ are  more than a distance $2g+1$ apart, the trees $\T^{j,i}$,
$j=1, 2, i = 1, \ldots, \Delta$ are non-intersecting and disconnected 
in $\G \sm \{u_1, u_2\}$ (see Figure ~\ref{figure:rewire}).    Therefore, by 
the spatial Markov property, 
\begin{align}  
\notag
 &\P_{\G \sm \{u_1,u_2\}}\left( X_{u_{1,i}}+X_{u_{2,i}}\leq 1, ~ 1\le i\le\Delta| \mbX_{\cup_{j=1}^2\cup_{i=1}^\Delta \partial \T^{j,i}} \right) \\
 &=  \prod_{i=1}^\Delta \P_{\cup_{j=1}^2 \partial \T^{j,i}} \left(X_{u_{1,i}}+X_{u_{2,i}}\leq
   1|\mbX_{\cup_{j=1}^2 \partial \T^{j,i}} \right) \notag \\
&= \prod_{i=1}^\Delta \int_{x_{u_{1,i}} \in [0,1]} \P_{\T^{1,i}} (d
  x_{u_{1,i}}| \mbX_{\partial \T^{1,i}} ) \P_{\T^{2,i}} (X_{u_{2,i}}
  \leq 1 - x_{u_{1,i}} |\mbX_{\partial \T^{2,i}}).  
\label{expr1}
\end{align}
Now, each tree $\T^{j,i}$ is a $\Delta$-regular rooted tree (recall that the nodes
$u_1$ and $u_2$ have been removed), with each node (other than the leaves) having $(\Delta-1)$
children, and is thus isomorphic to $\T_{g-1, \Delta-1}$.   
Therefore, recalling the definition of $F_{n,\Delta} (\cdot|\mbx_{\partial \T_{n,\Delta}})$
$ = F_{n,\Delta,  \lambda}(\cdot|\mbx_{\partial \T_{n,\Delta}})$ from  Section
\ref{sec-mainres}, for every
$\mbx_{\cup_{j=1}^2 \partial \T^{j,i}}  \in \Pol(\cup_{j=1}^2 \partial \T^{j,i} ) $, we have
\begin{align*} 
& \P_{\cup_{j=1}^2 \partial \T^{j,i}}\left(X_{u_{1,i}}+X_{u_{2,i}}\leq
   1|\mbX_{\cup_{j=1}^2 \partial \T^{j,i}} = \mbx_{\cup_{j=1}^2 \partial \T^{j,i}}  \right) \\
& \quad =  \int_{x_{u_{1,i}} \in   [0,1]} dF_{g-1, \Delta-1} (x_{u_{1,i}}|\mbx_{\partial\T^{1,i}})
F_{g-1,\Delta-1}(1 - x_{u_{1,i}}|\mbx_{\partial\T^{2,i}}). 
\end{align*}
Combining the last three displays with \eqref{eq:Ztilde} we obtain
\eqref{eq-ratio2}. 
\end{proof}

\appendix

\section{Proof of Lemma \ref{lemma:monotonicity}}
\label{ap-monotonicity}

\begin{proof}
To prove the lemma it clearly suffices to establish the following \\
{\em Claim: }  For every
$n$, and every two boundary conditions
$\mbx_{\partial\T_n}$, $\mby_{\partial\T_n} \in [0,1]^{\partial \T_n}$
such that $\mbx_{\partial\T_n}\le \mby_{\partial\T_n}$
coordinate-wise, there exist 
random variables $X$ and $Y$ such that $\pr(X\le \ex)=F_n(\ex|\mbx_{\partial\T_n}),\pr(Y\le \ex)=F_n(\ex|\mby_{\partial\T_n}), \ex\in [0,1]$
and almost surely $X\le Y$ when $n$ is even and $X\ge Y$ when $n$ is
odd.

We establish the claim by induction on $n$, and repeatedly use
the following elementary observation regarding the coupling of two random variables with the same distribution: 
given a random variable $U$ with cumulative  distribution function $F$
and two real numbers $\theta_1<\theta_2$, there exists a probability
space and a random vector $(X_1, X_2)$ defined on it  such that $X_i$ has the distribution of $U$ conditioned on $X \le \theta_i, i=1,2,$ and $X_1\le X_2$
almost surely.  In what follows, let $U$ be distributed according to the free spin
measure $\mu$. 
We now prove the claim for $n=1$. Given $\mbx_{\partial\T_1}$, $\mby_{\partial \T_1}$,
let $\bar{x} = \max_{i \in \partial \T_1} (\mbx_{\partial \T_1})_i$ and $\bar{y} =
\max_{i \in \partial T_1} (\mby_{\partial \T_1})_i$.  Then, by the
definition of the hardcore model, $F_1(\cdot|\mbx_{\partial\T_1})$ is 
equal to the conditional distribution of $U$ given $U \leq 1- \bar{x}$
and likewise $F_1 (\cdot|\mbx_{\partial T_1})$ is the conditional
distribution of $U$ given $U \leq 1- \bar{y}$. Since
$\mbx_{\partial\T_1} \leq \mby_{\partial \T_1}$ implies $1  - \bar x
\geq 1 - \bar y$, the claim for $n=1$ follows from the 
observation made above.

Now, for the induction step, assume the claim holds for $n= 1, \ldots,
m-1$. Suppose $m$ is even.  Consider two copies of the tree $\T_m$, 
with roots $u$ and $v$ respectively, and label their children as 
   $u_1,\ldots,u_\Delta,$  and $v_1,\ldots,v_\Delta$, respectively. 
On these two copies consider two 
arbitrary boundary conditions $\mbx_{\partial\T_m}$ and
$\mby_{\partial\T_m}$, respectively, that satisfy 
$\mbx_{\partial\T_m}\le \mby_{\partial\T_m}$. 
 For $i = 1, \ldots, \Delta$, let $\mbx_{\partial\T_m}^i$
 (respy, $\mby_{\partial\T_m}^i$)  be the
 natural restriction of the boundary condition $\mbx_{\partial\T_m}$
 (respy, $\mby_{\partial\T_m}$) 
 to the subtree corresponding to $u_i$ (respy, $v_i$), each of which is a 
copy of the tree $\T_{m-1}$.
By the inductive assumption, since $m-1$ is odd,
for each $i=1,\ldots,\Delta$, there exist two coupled random variables
$X_i$ and $Y_i$ distributed according to
$F_{m}(\cdot|\mbx_{\partial\T_m}^i)$ and
$F_{m}(\cdot|\mby_{\partial\T_m}^i)$, respectively,  such that $X_i\ge Y_i$ almost surely. Generate pairs $(X_i,Y_i)$
independently across $i=1,\ldots,\Delta$ in this way. Now let $U$ be a
random variable distributed according to 
the free spin measure $\mu$.  Then the random variable 
$X$ distributed according to  $F_{m}(\cdot|\mbx_{\partial\T_m})$ 
has the conditional distribution of 
$U$ given   $U \le 1-\max_{1\le i\le \Delta}X_i$, integrated over the
joint distribution of  $X_1,\ldots,X_\Delta$. 
Similarly, $Y$ distributed according to $F_{m}(\cdot|\mbx_{\partial\T_m})$ is distributed as the
conditional distribution of $U$ given   $U \le 1-\max_{1\le i\le \Delta}Y_i$, integrated over the
joint distribution of  $Y_1,\ldots, Y_\Delta$. 
Since by construction we have $X_i\ge Y_i$, then $1-\max_{1\le i\le \Delta}X_i\le 1-\max_{1\le i\le \Delta}Y_i$. Thus, there
exists a coupling of $X$ and $Y$ such that $X\le Y$ almost
surely.  The case of odd $m$ is analyzed similarly, using that the
result holds for all even $n < m$. Hence, the details 
are omitted.   The claim then follows by induction. 
\end{proof}

\section{Proof of Lemma \ref{lemma:Rewire}} 
\label{sec-ap1}

\begin{proof}[Proof of Lemma \ref{lemma:Rewire}]   

In every step of the rewiring we delete two nodes in the graph. Thus, when
we perform $t\leq (N/2)-(2g+1)\Delta^{2g}$ rewiring steps sequentially, in the end
we obtain a graph with at least $N-2((N/2)-(2g+1)\Delta^{2g})=2(2g+1)\Delta^{2g}$ nodes.
Suppose that at  step $t\leq (N/2)-(2g+1)\Delta^{2g}$ we have a graph
$\G_t$ that  is $\Delta$-regular and has girth at least $g$.
We claim that the diameter of this graph is at least $2g+1$. Indeed, if the diameter is smaller than $2g+1$, then for any given
node $v$ any other node is reachable from $v$ by a path with length at most $2g$ and thus the total number of
nodes is at most $\sum_{0\leq k\leq 2g}\Delta^k<(2g+1)\Delta^{2g}$, which is a contradiction, and the claim is established.

Now, given any $t\leq (N/2)-(2g+1)\Delta^{2g}$,
suppose the rewiring was performed at least $t$ steps on pairs of nodes with distance at least $2g+1$ apart.
Select any two nodes $u_1,u_2$ in the resulting graph $\G_t$ which are at
the distance equal to the diameter
of $\G_t$, and thus are at least $2g+1$ edges apart. We already showed that the graph $\G_{t+1}$ obtained by rewiring $\G_t$ on
$v_1,v_2$ is $\Delta$-regular.
It remains to show it has  girth at least $g$. Suppose, for the purposes of contradiction, $\G_{t+1}$ has girth $\leq g-1$.
Denote by $u_{j,1},\ldots,u_{j,\Delta}$ the $\Delta$ neighbors of $u_j,~j=1,2$.
Suppose $k\geq 1$ is the number of
newly created edges which participate in creating a cycle with length $\leq g-1$. If $k=1$ and $u_{1,j},u_{2,j}$ is the pair
creating the unique participating
edge, then the original distance between $u_{1,j}$ and $u_{2,j}$ was at most $g-2$ by following a path on the cycle that does not
use the new edge.
But then the distance between $u_1$ and $u_2$ is at most $g<2g+1$,
which gives a contradiction.
Now suppose $k>1$, then there exists a path of length at most $(g-1)/k\leq (g-1)/2$ which uses only the original edges
(the edges of the graph $\G_t$)
and connects a pair  $v,v'$ of nodes from the set $u_{1,1},\ldots,u_{1,\Delta},u_{2,1},\ldots,u_{2,\Delta}$. If the pair is from the same set,
say $u_{1,1},\ldots,u_{1,\Delta}$,
then, since these two nodes are connected to $u_1$, we obtain a cycle in $\G_t$ with length $(g-1)/2+2<g$, leading to a contradiction
(since by assumption $g>3$).
If these two nodes
are from different sets, for example $v=u_{1,j},v'=u_{2,l}$, then we obtain that the
distance between $u_1$ and $u_2$ in $G_t$ is at most $(g-1)/2+2<2g+1$,which also leads to
a contradiction. So we conclude that $\G_{t+1}$ must have girth at
least $g$, as stated. 
\end{proof}

\end{document}